%% file: 20240417splitting.tex
\documentclass{amsart}
\usepackage[ascii]{inputenc}
\usepackage{lineno}

\usepackage{amssymb,comment,dsfont,subfig}
\usepackage{bbm}
\usepackage{bm}
\usepackage[dvipsnames]{xcolor} 
\usepackage{graphicx,enumerate} 
\usepackage[all]{xy} 
\usepackage{tikz}
\usetikzlibrary{decorations.pathreplacing}
\usepackage{mathtools} 
\mathtoolsset{centercolon} 
\xyoption{rotate}
\xyoption{pdf}
\newtheorem{theorem}[equation]{Theorem}
\newtheorem{lemma}[equation]{Lemma}

\newtheorem{goal}[equation]{Goal}

\newtheorem{fact}[equation]{Fact}
\newtheorem{corollary}[equation]{Corollary}
\newtheorem{question}{Question}
\newtheorem*{mainthm}{Main Theorem}
\theoremstyle{definition}
\newtheorem{deflemma}[equation]{Definition and Lemma}
\newtheorem{definition}[equation]{Definition}
\newtheorem{remark}[equation]{Remark}
\newtheorem*{remark*}{Remark}
\newtheorem{example}[equation]{Example}
\newtheorem{assumption}[equation]{Assumption}
\newtheorem{notation}[equation]{Notation}
\numberwithin{equation}{section}

\newcommand{\Hgroup}{\mathcal H} 
\usepackage[hidelinks]{hyperref}

\DeclareMathOperator{\cf}{cof}

\newcommand{\COB}{\mathsf{COB}}
\newcommand{\LCU}{\mathsf{LCU}}

\DeclareMathOperator{\add}{add}
\DeclareMathOperator{\cov}{cov}
\DeclareMathOperator{\non}{non}
\DeclareMathOperator{\cof}{cof}

\DeclareMathOperator{\cp}{cp}

\DeclareMathOperator{\dom}{dom}
\DeclareMathOperator{\ran}{ran}

\newcommand{\Null}{\mathcal N}
\newcommand{\Meager}{\mathcal M}
\newcommand{\addN}{{\ensuremath{\add(\Null)}}}
\newcommand{\cofN}{{\ensuremath{\cof(\Null)}}}
\newcommand{\covN}{{\ensuremath{\cov(\Null)}}}
\newcommand{\nonN}{{\ensuremath{\non(\Null)}}}
\newcommand{\addM}{{\ensuremath{\add(\Meager)}}}
\newcommand{\cofM}{{\ensuremath{\cof(\Meager)}}}
\newcommand{\covM}{{\ensuremath{\cov(\Meager)}}}
\newcommand{\nonM}{{\ensuremath{\non(\Meager)}}}

\newcommand{\cfrak}{\mathfrak{c}}
\newcommand{\bfrak}{\mathfrak{b}}
\newcommand{\dfrak}{\mathfrak{d}}
\newcommand{\gfrak}{\mathfrak{g}}
\newcommand{\hfrak}{\mathfrak{h}}
\newcommand{\sfrak}{\mathfrak{s}}
\newcommand{\rfrak}{\mathfrak{r}}
\newcommand{\mfrak}{\mathfrak{m}}
\newcommand{\pfrak}{\mathfrak{p}}
\newcommand{\tfrak}{\mathfrak{t}}
\newcommand{\xfrak}{\mathfrak{x}}
\DeclareMathOperator{\mlabel}{m} 
\DeclareMathOperator{\plabel}{p} 

\newcommand{\mlike}{$\mathfrak{m}$-like}
\newcommand{\hlike}{$\mathfrak{h}$-like}

\newcommand{\innitialmark}[1]{{#1}^\text{pre}}

\newcommand{\addNi}{\innitialmark{\addN}}
\newcommand{\covNi}{\innitialmark{\covN}}

\newcommand{\nonMi}{\innitialmark{\nonM}}

\newcommand{\bfraki}{\innitialmark{\bfrak}}

\newcommand{\cfraki}{\innitialmark{\cfrak}}

\newcommand{\la}{\langle}
\newcommand{\ra}{\rangle}
\newcommand{\seq}[2]{\la #1 : #2 \ra}
\newcommand{\set}[2]{\{ #1 : #2 \}}

\AtBeginDocument{%
   \def\MR#1{}
}

\title[Preservation of splitting families]{Preservation of splitting families \\ and cardinal characteristics of the continuum}

\author[M.\ Goldstern]{Martin Goldstern}
\address{Institut f\"ur Diskrete Mathematik und Geometrie, TU Wien, 1040 Vienna, Austria.}
\email{martin.goldstern@tuwien.ac.at}
\urladdr{http://www.tuwien.ac.at/goldstern/}

\author[J.\ Kellner]{Jakob Kellner}
\address{Institut f\"ur Diskrete Mathematik und Geometrie, TU Wien, 1040 Vienna, Austria.}
\email{kellner@fsmat.at}
\urladdr{http://dmg.tuwien.ac.at/kellner/}

\author[D.A.\ Mej\'{i}a]{Diego A. Mej\'{i}a}
\address{Creative Science Course (Mathematics), Faculty of Science, Shizuoka University, Ohya 836, Suruga-ku, Shizuoka-shi, Japan 422-8529.}
\email{diego.mejia@shizuoka.ac.jp}
\urladdr{http://www.researchgate.net/profile/Diego\_Mejia2}

\author[S.\ Shelah]{Saharon Shelah}
\address{Einstein Institute of Mathematics, The Hebrew University of Jerusalem, Jerusalem, 91904, Israel, and Department of Mathematics, Rutgers University, New Brunswick, NJ 08854, USA.}
\email{shelah@math.huji.ac.il}
\urladdr{http://shelah.logic.at}

\thanks{This work was supported by the following grants:
Austrian Science Fund (FWF): project number I3081, P29575
(first author); P30666
(second author);
Grant-in-Aid for Early Career Scientists 18K13448, Japan Society for the Promotion of Science (third author); Israel Science Foundation (ISF) grant no: 1838/19 (fourth author). This is publication number 1199 of the fourth author.}

\date{\today}
\DeclareMathOperator{\supp}{supp}

\begin{document}

\begin{abstract}
    We show how to construct, via forcing, splitting families than are preserved by a certain type
    of finite support iterations. As an application, we construct a model where 15 classical characteristics of the continuum are pairwise different, concretely: the 10 (non-dependent) entries in Cicho\'n's diagram, $\mfrak(2\text{-Knaster})$, $\pfrak$, $\hfrak$, the splitting number $\sfrak$ and the reaping number $\rfrak$.
\end{abstract}

\maketitle

\newcommand{\Bbf}{\mathbf{B}}
\newcommand{\Dbf}{\mathbf{D}}
\newcommand{\Rbf}{\mathbf{R}}
\newcommand{\tbf}{{}}
\newcommand{\ttbf}{{\Por}}

\newcommand{\Bwf}{\mathcal{B}}
\newcommand{\Dwf}{\mathcal{D}}
\newcommand{\Fwf}{\mathcal{F}}
\newcommand{\Mwf}{\mathcal{M}}
\newcommand{\Pwf}{\mathcal{P}}

\newcommand{\Cor}{\mathbb{C}}
\newcommand{\Gor}{\mathbb{G}}
\newcommand{\Hor}{\mathbb{H}}
\newcommand{\Por}{\mathbb{P}}
\newcommand{\Qor}{\mathbb{Q}}
\newcommand{\R}{\mathbb{R}}
\newcommand{\Mor}{\mathbb{M}}
\newcommand{\Sor}{\mathbb{S}}
\newcommand{\Zbb}{\mathbb{Z}}

\newcommand{\Qnm}{\dot{\mathbb{Q}}}

\newcommand{\proj}{\mathrm{pr}}
\newcommand{\Rsp}{R_{\mathrm{sp}}}
\newcommand{\Rbsp}{\mathbf{R}_{\mathrm{sp}}}

\newcommand{\menos}{\smallsetminus}
\newcommand{\imp}{{\ \mbox{$\Rightarrow$} \ }}
\newcommand{\sii}{{\ \mbox{$\Leftrightarrow$} \ }}
\newcommand{\id}{\mathrm{id}}
\newcommand{\leqT}{\leq_{\mathrm{T}}}
\newcommand{\eqT}{\cong_{\mathrm{T}}}

\newcommand{\blue}[1]{{\color{blue}#1}}
\newcommand{\red}[1]{{\color{red}#1}}


\section{Introduction}

In this paper we present a method to preserve certain  splitting families along finite support iterations.
These splitting families are constructed via forcing, using specific uncountable $2$-edge-labeled graphs\footnote{A \emph{$2$-edge-labeled graph} is a simple graph whose edges are labeled by either $0$ or $1$.} as support.
The main application of this method is a forcing model where many classical cardinal characteristics of the continuum
are pairwise different, including the \emph{splitting number} $\sfrak$ and the \emph{reaping number} $\rfrak$.

We assume that the reader is familiar with \emph{Cicho\'n's diagram} (Figure~\ref{fig:cichon}) containing the characteristics that we will call
\emph{Cicho\'n-characteristics}.
We also investigate some of the characteristics in the \emph{Blass diagram}~\cite[Pg.~481]{Blass}. Figure~\ref{fig:all20} illustrates both diagrams combined, along with all the ZFC-provable inequalities that we are aware of. See~\cite{Blass,BaJu} for the definitions and the proofs for the inequalities (with the exception of $\cofM\leq\mathfrak{i}$, which was proved in~\cite{BHH}). In the following, we only give the definitions of the non-Cicho\'n-characteristics that we will investigate in this paper.

\begin{definition}\label{def:cardchar}

   \begin{enumerate}[(1)]
     \item For $a,b\in[\omega]^{\aleph_0}$, we define $a\subseteq^* b$ iff $a\smallsetminus b$ is finite;
     \item and we say that \emph{$a$ splits $b$} if both $a\cap b$ and $b\smallsetminus a$ are infinite, that is, $a\nsupseteq^* b$ and $\omega\smallsetminus a\nsupseteq^* b$.

    \item $F\subseteq[\omega]^{\aleph_0}$ is a \emph{splitting family} if every $y\in[\omega]^{\aleph_0}$ is split by some $x\in F$. The \emph{splitting number} $\sfrak$ is the smallest size of a splitting family.
    \item $D\subseteq[\omega]^{\aleph_0}$ is an \emph{unreaping family} if no $x\in[\omega]^{\aleph_0}$ splits every member of $D$. The \emph{reaping number} $\rfrak$ is the smallest size of an unreaping family.
    \item $D\subseteq[\omega]^{\aleph_0}$ is \emph{groupwise dense} when:
      \begin{enumerate}[(i)]
        \item if $a\in[\omega]^{\aleph_0}$, $b\in D$ and $a\subseteq^* b$, then $a\in D$,
        \item if $\la I_n:n<\omega\ra$ is an interval partition of $\omega$ then $\bigcup_{n\in a}I_n\in D$ for some $a\in[\omega]^{\aleph_0}$.
      \end{enumerate}
      The \emph{groupwise density number $\gfrak$} is the smallest size of a collection of groupwise dense sets whose intersection is empty.
    \item The \emph{distributivity number $\hfrak$} is the smallest size of a collection of dense subsets of $\la[\omega]^{\aleph_0},\subseteq^*\ra$ whose intersection is empty.
    \item Say that $a\in[\omega]^{\aleph_0}$ is a \emph{pseudo-intersection of $F\subseteq[\omega]^{\aleph_0}$} if $a\subseteq^* b$ for all $b\in F$.
    \item The \emph{pseudo-intersection number} $\pfrak$ is the smallest size of a filter base of subsets of $[\omega]^{\aleph_0}$ without pseudo-intersection.
    \item The \emph{tower number} $\tfrak$ is the smallest length of a (transfinite) $\subseteq^*$-decreasing sequence in $[\omega]^{\aleph_0}$ without pseudo-intersection.
    \item Given a class $\Pwf$ of forcing notions, $\mfrak(\Pwf)$ denotes the minimal cardinal $\kappa$ such that, for some $Q\in\Pwf$, there is some collection $\Dwf$ of size $\kappa$ of dense subsets of $Q$ without a filter in $Q$ intersecting every member of $\Dwf$.
    \item Let $\Por$ be a poset. A set $A\subseteq \Por$ is \emph{$k$-linked (in $\Por$)} if every $k$-element subset of $A$
    has a  lower bound in $\Por$.
    $A$ is \emph{centered} if it is $k$-linked for all $k\in\omega$.
    \item
    A poset $\Por$ is \emph{$k$-Knaster}, if for each uncountable $A\subseteq \Por$
    there is a $k$-linked uncountable $B\subseteq A$.
    And $\Por$ \emph{has precaliber $\aleph_1$}, if such a $B$ can be chosen centered.
    For notational convenience, \emph{$1$-Knaster} means ccc, and \emph{$\omega$-Knaster} means precaliber $\aleph_1$.
    \item
    For $1\leq k\leq \omega$ denote $\mfrak_k:=\mfrak(k\text{-Knaster})$ and $\mfrak:=\mfrak_1$. We also set $\mfrak_0:=\aleph_1$.
  \end{enumerate}
\end{definition}

\newcommand{\mye}{*+[F.]{\phantom{\lambda}}}
\begin{figure}
\resizebox{\textwidth}{!}{$
\xymatrix@=2ex{
&
&
&\cfrak
\\
\covN\ar[r]
&\nonM\ar[r]
&\cofM\ar[r]
&\cofN\ar[u]
\\
&
\mathfrak b\ar[r]\ar[u]
&\mathfrak d\ar[u]
\\
\addN\ar[r]\ar[uu]
&\addM\ar[r]\ar[u]
&\covM\ar[r]\ar[u]
&\nonN\ar[uu]
\\
\aleph_1\ar[u]
}\quad
\xymatrix@=2ex{
&
&
&\cfrak
\\
\covN\ar[r]
&\nonM\ar[r]
&\mye\ar[r]
&\cofN\ar[u]
\\
&
\mathfrak b\ar[r]\ar[u]
&\mathfrak d\ar[u]
\\
\addN\ar[r]\ar[uu]
&\mye\ar[r]\ar[u]
&\covM\ar[r]\ar[u]
&\nonN\ar[uu]
\\
\aleph_1\ar[u]
}
$}
    \caption{\label{fig:cichon}Cicho\'n's diagram (left). In the version on the right, the
two ``dependent'' values $\addM=\min\{\mathfrak b, \covM\}$ and $\cofM=\max\{\nonM,\mathfrak d\}$ are removed; the ``independent'' ones remain (nine entries excluding $\aleph_1$, or ten including it). An arrow $\mathfrak x\rightarrow \mathfrak y$ means that ZFC proves $\mathfrak x\le \mathfrak y$.}
\end{figure}
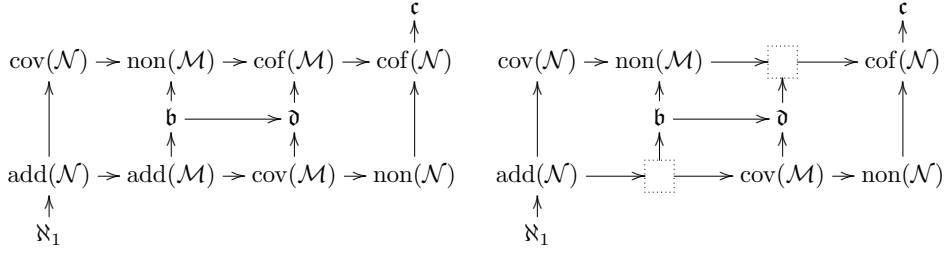

\usetikzlibrary{arrows}

\begin{figure}
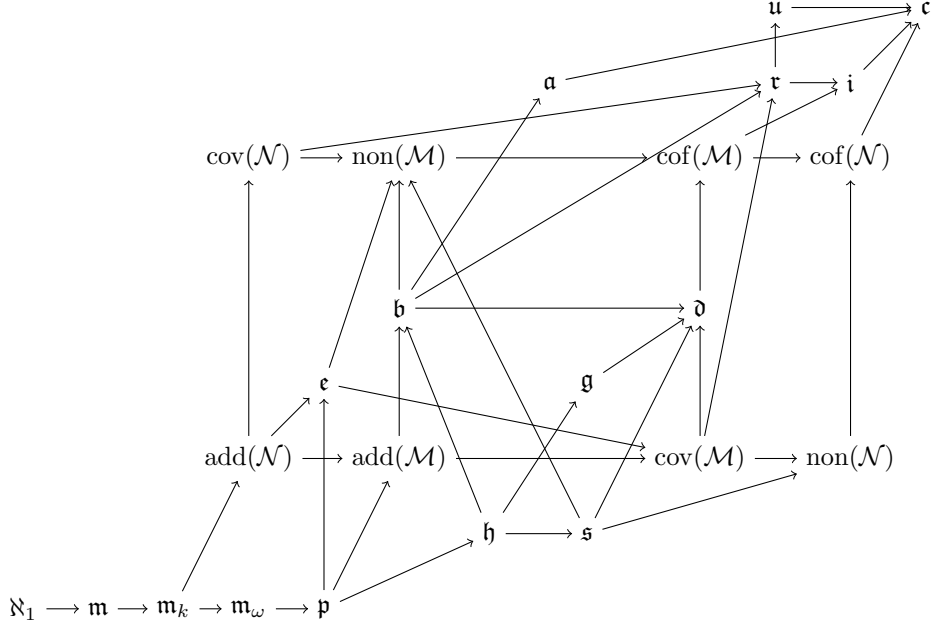

\[
\tikz{
\node (a1) at (-2,1) {$\aleph_1$} ;
\node (m) at (-1,1) {$\mathfrak m$} ;
\node (mk) at (0,1) {$\mathfrak{m}_k$} ;
\node (prec) at (1,1) {$\mathfrak{m}_\omega$} ;
\node (p) at (2,1) {$\mathfrak p$} ; 
\node (e) at (2,4) {$\mathfrak e$};
\node (addn) at (1,3){\addN};
\node (covn) at (1,7){\covN};
\node (nonn) at (9,3) {\nonN} ;
\node (cfn) at (9,7) {\cofN} ;
\node (addm) at (3,3) {\addM} ;
\node (covm) at (7,3) {\covM} ;
\node (nonm) at (3,7) {\nonM} ;
\node (cfm) at (7,7) {\cofM} ;
\node (b) at (3,5) {$\mathfrak b$};
\node (d) at (7,5) {$\mathfrak d$};
\node (h) at (4.2,2)   {$\mathfrak h$};
\node (s) at (5.5,2)   {$\mathfrak s$};
\node (g) at (5.5,4)   {$\mathfrak g$};
\node (a) at (5,8)   {$\mathfrak a$};
\node (r) at (8,8)   {$\mathfrak r$};
\node (u) at (8,9)   {$\mathfrak u$};
\node (i) at (9,8)   {$\mathfrak i$};
\node (c) at (10,9) {$\mathfrak c$};
\draw (a1) edge[->] (m);
\draw (m) edge[->] (mk);
\draw (mk) edge[->] (addn)
     (addn) edge[->] (covn)
     (covn) edge [->] (nonm)
      (nonm)edge [->] (cfm)
      (cfm)edge [->] (cfn);
\draw
   (addn) edge [->]  (addm)
   (addm) edge [->]  (covm)
   (covm) edge [->]  (nonn)
   (nonn) edge [->]  (cfn);
\draw (addm) edge [->] (b)
     (b)  edge [->] (nonm);
\draw (covm) edge [->] (d)
      (d)  edge[->] (cfm);
\draw (b) edge [->] (d);
\draw (mk) edge[->] (prec)
     (prec) edge[->] (p)
     (p) edge[->] (e);
\draw (p) edge [->] (h);
\draw (h) edge [->] (b);
\draw (b) edge [->] (r);
\draw (r) edge [->] (u);
\draw (covn) edge [->] (r);
\draw (r) edge [->] (i);
\draw (b) edge [->] (a);
\draw (e) edge [->] (nonm);
\draw (h) edge [->] (g);
\draw (g) edge [->] (d);
\draw (h) edge [->] (s);
\draw (s) edge [->] (d);
\draw (s) edge [->] (nonm);
\draw (s) edge [->] (nonn);
\draw (cfm) edge [->] (i);
\draw (addn) edge [->] (e)  (e) edge [->] (covm);
\draw (covm) edge [->] (r);
\draw (p) edge [->] (addm);
\draw (cfn) edge[->] (c);
\draw (i) edge[->] (c);
\draw (u) edge[->] (c);
\draw (a) edge[->] (c);
}
\]
\caption{\label{fig:all20}Cicho\'n's diagram and the Blass  diagram combined. An arrow $\mathfrak x\rightarrow \mathfrak y$ means that ZFC proves $\mathfrak x\le \mathfrak y$.}
\end{figure}

Below we list some additional properties of these cardinals.
Unless noted otherwise, proofs can be found in~\cite{Blass}.
\begin{fact}\label{fact:blass}
\begin{enumerate}
  \item In~\cite{MSpt} it was proved that $\pfrak=\tfrak$.\footnote{Only the trivial inequality $\pfrak\leq\tfrak$ is used in this text.}
  \item The cardinals $\addN$, $\addM$, $\bfrak$, $\tfrak$, $\hfrak$ and $\gfrak$ are regular.
  \item $\cf(\sfrak)\geq\tfrak$ (see \cite{DowShelah}).
  \item $2^{{<}\tfrak}=\cfrak$.
  \item $\cf(\cfrak)\geq\gfrak$.
  \item For $1\leq k\leq k'\leq\omega$, $\mfrak_k\leq\mfrak_{k'}$.
  \item\label{mart} For $1\leq k\leq\omega$, $\mfrak_k>\aleph_1$ implies $\mfrak_k=\mfrak_\omega$ (well-known, but see e.g.\ \cite[Lemma~4.2]{GKMS1}).
\end{enumerate}
\end{fact}
This work contributes to the 
project of constructing a forcing model satisfying:
\begin{equation}\tag{$\heartsuit$}\label{eq:maingoal}
    \text{All the cardinals in Figure~\ref{fig:all20} are pairwise different,}
\end{equation} with the obvious (ZFC provable) exception of the dependent entries $\addM=\min\{\bfrak,\covM\}$ and $\cofM=\max\{\nonM,\dfrak\}$, and the Martin axiom numbers $\mfrak$, $\mfrak_k$ for some $2\le k<\omega$,
and $\mfrak_\omega$, which can not have more than one value ${>}\aleph_1$, see Fact~\ref{fact:blass}(\ref{mart}).

In this direction \cite{GKS} constructed a forcing model, using four strongly compact cardinals, where all the ten (non-dependent) values of Cicho\'n's diagram are pairwise different (a situation we call \emph{Cicho\'n's Maximum}), as in Figure~\ref{fig:cichonorders}(A). This was improved later in~\cite{diegoetal} by only using three strongly compact cardinals; finally
in~\cite{GKMS2} it was shown that no large cardinals are needed for Cicho\'n's Maximum.

A model of Cicho\'n's Maximum with the order as in Figure~\ref{fig:cichonorders}(B) was obtained in~\cite{KeShTa:1131}. Although this model initially required four strongly compact cardinals as well, the methods of~\cite{GKMS2} allow to remove the large cardinal assumptions also here.

As a next step towards~\eqref{eq:maingoal}, \cite{GKMS1} proved:

\begin{theorem}[{\cite{GKMS1}}]\label{GKMSmain}
   Under $\mathrm{GCH}$, for any $k\in[1,\omega)$,
   there is a cofinality preserving poset  $\Por_{k}$ forcing that
   \begin{enumerate}[(a)]
     \item Cicho\'n's Maximum holds with the order of Figure~\ref{fig:cichonorders}\textnormal{(\textsc{a})}.
     \item $\aleph_1=\mfrak_{k-1}<\mfrak_k=\mfrak_\omega<\pfrak<\hfrak<\addN$ (recall $\mfrak_0:=\aleph_1$).
   \end{enumerate}
   An analogous result holds for the alternative order of Figure~\ref{fig:cichonorders}\textnormal{(\textsc{b})}.
\end{theorem}

In this paper, we continue this line of work by including, in addition, $\sfrak$ and~$\rfrak$.

\begin{mainthm}
   Under $\mathrm{GCH}$, for any $k\in[2,\omega)$ there is a cofinality preserving poset forcing that the cardinals in Cicho\'n's diagram, $\mfrak_k$, $\pfrak$, $\hfrak$, $\sfrak$ and $\rfrak$ are pairwise different. More specifically:
   \begin{enumerate}[(a)]
      \item Cicho\'n's Maximum holds, in either of the orders of Figure~\ref{fig:cichonorders}.
      \item $\aleph_1=\mfrak_{k-1}<\mfrak_k=\mfrak_\omega<\pfrak<\hfrak<\addN$.
      \item $\sfrak$ can assume any regular value between $\pfrak$ and $\bfrak$.
      \item $\rfrak$ can assume any regular value in the dual position to $\sfrak$. E.g., if $\sfrak<\addN$, then $\rfrak$ can be any arbitrary regular in $[\cofN,\cfrak]$ (see details in Section~\ref{sec:15}).
   \end{enumerate}
\end{mainthm}

In both theorems above, item (b) can also be replaced by $\aleph_1<\mfrak_\omega<\pfrak<\hfrak<\addN$ while $\mfrak_k=\aleph_1$ for all $k<\omega$. Those are the only possible constellations of the Knaster numbers, by Fact~\ref{fact:blass}(\ref{mart}), unless you count $\mfrak$ as the $1$-Knaster-number:
In contrast to Theorem~\ref{GKMSmain} (where we do not control $\rfrak,\sfrak$),
we cannot force $\mfrak>\aleph_1$ with the methods we use here.
We cannot just iterate over all small ccc forcings one by one
to increase $\mfrak$, as our method requires
that all iterands of the forcing iteration have to be ``homogeneous''.
So instead of using a certain small forcing $\dot Q$ as iterand, we will use
a finite support product over all variants as iterand.
So only if $\dot Q$ (and therefore all variants) is Knaster,\footnote{Or at least stays ccc in ccc extensions.} this product can be used
in a ccc iteration; accordingly we can increase the Knaster numbers but not $\mfrak$ itself.

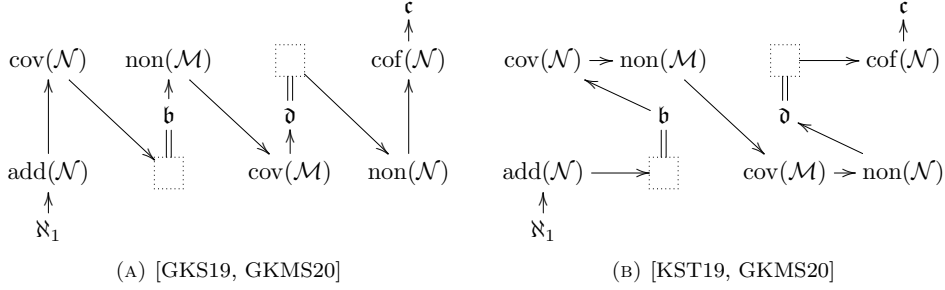
\begin{figure}
\subfloat[\cite{GKS,GKMS2}]
{\resizebox{0.48\textwidth}{!}{$
\xymatrix@=2ex{
&
&
&\cfrak
\\
\covN\ar[ddr]
&\nonM\ar[ddr]
&\mye\ar[ddr]
&\cofN\ar[u]
\\
&
\mathfrak b\ar[u]
&\mathfrak d\ar@{=}[u]
\\
\addN\ar[uu]
&\mye\ar@{=}[u]
&\covM\ar[u]
&\nonN\ar[uu]
\\
\aleph_1\ar[u]
}
$}}
\hspace*{\fill}
\subfloat[\cite{KeShTa:1131,GKMS2}]
{\resizebox{0.48\textwidth}{!}{$
\xymatrix@=2ex{
&
&
&\cfrak
\\
\covN\ar[r]
&\nonM\ar[ddr]
&\mye\ar[r]
&\cofN\ar[u]
\\
&
\mathfrak b\ar[ul]
&\mathfrak d\ar@{=}[u]
\\
\addN\ar[r]
&\mye\ar@{=}[u]
&\covM\ar[r]
&\nonN\ar[ul]
\\
\aleph_1\ar[u]
}
$}}
    \caption{\label{fig:cichonorders}The two known consistent orders where all the (non-dependent) values in Cicho\'n's diagram are pairwise different. (A) corresponds to the model in~\cite{GKS}, and (B) to the model in~\cite{KeShTa:1131} (both proven consistent in~\cite{GKMS2} without large cardinals). Each arrow can be $<$ or $=$ as desired.}
\end{figure}

We remark that the full power of GCH is not required in the Main Theorem, but we do need \emph{some} assumption on cardinal arithmetic in the ground model.
See details in Section~\ref{sec:15}.

In order to include $\sfrak$ and $\rfrak$ in our main result, we need a new preservation theorem for splitting families. Previously, the following was known in the context of FS (finite support) iterations:
\begin{itemize}
  \item[\cite{baudor}] Hechler forcing (for adding a dominating real) preserves splitting families witnessing the property $\LCU_{\Rbsp}(\kappa)$ for any uncountable regular $\kappa$ (see Section~\ref{sec:COB}).
  \item[\cite{JSsuslin}] Assuming CH, any FS iteration of Suslin ccc posets forces that the ground model reals form a splitting family.
\end{itemize}
In this paper we will use a splitting family obtained by a FS product
of Hechler-type posets (cf.~\cite{Hechlermad}) which we call $\Gor_\Bbf$;
the support of $\Gor_\Bbf$ is a graph $\Bbf$ of size $\aleph_1$ with certain homogeneity properties.
We then show that this
splitting family is preserved by certain FS iterations, which we will call ``symmetric Suslin-$\lambda$-small''. (Every FS iteration of Suslin ccc posets with parameters in the ground model is such an iteration, but our application will not use such ``full'' Suslin ccc forcings.)

Similar preservation techniques have appeared in different contexts. For instance, concerning preservation of mad (maximal almost disjoint) families, Kunen~\cite{kunen80} constructed, under CH, a mad family that can be preserved by Cohen posets; afterwards, Steprans~\cite{steprans} showed that, after adding $\omega_1$-many Cohen reals, there is a mad family of size $\aleph_1$ that can be preserved in further Cohen extensions; Fischer and Brendle~\cite{VFJB11} constructed a Hechler-type poset $\Hor_A$ with support (any uncountable set) $A$ that adds a mad family indexed by $A$, which can be preserved not only in further Cohen extensions but after other concrete FS iterations, thus generalizing Steprans' result because $\Hor_{\omega_1}=\Cor_{\omega_1}$; \cite{FFMM,mejiavert} showed that any such mad family added by $\Hor_A$ can be preserved by some general type of FS iterations, but the most general result so far was shown in~\cite{diegoetal}: Any $\kappa$-$\mathrm{Fr}$-Knaster poset preserves $\kappa$-strong-$\mathbf{Md}$-families (with $\kappa$ uncountable regular; the mad family added by $\Hor_\kappa$ is of such type).

There are deep technical differences between the mad family added by this $\Hor_A$,
and the construction of a splitting family in this paper: No structure is needed on $A$, and because of this it is clear that Hechler's posets satisfy  $\Hor_A\lessdot\Hor_B$ whenever $A\subseteq B$; but we cannot guarantee $\Gor_{\Bbf_0}\lessdot\Gor_\Bbf$ for our posets, whenever  $\Bbf_0$ is a  subgraph of~$\Bbf$. Also, $\Gor_\Bbf$ itself does not add a splitting family, but it just adds a set of Cohen reals $\{\eta_a : a\in \Bbf\}$ \emph{over the ground model} (recall that we do not have intermediate extensions by restricting the support $\Bbf$). Hence, the FS product (or iteration, which is the same, as the poset $\Gor_\Bbf$ is absolute) of size $\kappa$ of such posets adds a splitting family of size $\kappa$ (witnessing $\LCU_{\Rbf}(\kappa)$) formed by the previously mentioned Cohen reals. It is clear that just adding $\kappa$ many Cohen reals produces a splitting family satisfying $\LCU_{\Rbf}(\kappa)$, but we need to use FS support products of
$\kappa$ many $\Gor_\Bbf$ (with $\Bbf$ of size $\aleph_1$, instead of just one $\Gor_{\Bbf'}$ with $\Bbf'$ of size $\kappa$), and we need the graph structure on $\Bbf$, to be able to guarantee the preservation of the new splitting family. The forcing structure is very important here because an isomorphism of names argument is required for this preservation.

The strategy to prove the main theorem is similar to Theorem~\ref{GKMSmain}. We first show how to construct a ccc poset that forces distinct values for the cardinals on the left side of Cicho\'n's diagram, including some of the other cardinal characteristics (like $\sfrak$ in this case). Afterwards, methods from~\cite{GKMS2,GKMS1} are applied to this initial forcing to get the poset for the main theorem.

\goodbreak
\subsection*{Acknowledgment}
We thank Cezar Port for pointing out a mistake in the published version of the proof of Theorem~\ref{SbGexists} (which is fixed in this version).

\subsection*{Annotated contents.}
\ \smallskip

\noindent\textbf{\S\ref{sec:graph}} We show how to construct, in ZFC, a \emph{suitable $2$-graph}. This is the type of graph we use as support for $\Gor_\Bbf$.\smallskip

\noindent\textbf{\S\ref{sec:COB}} The $\LCU$ and $\COB$ properties are reviewed from \cite{GKS,GKMS2,GKMS1}. These describe \emph{strong witnesses} to cardinal characteristics associated with a definable relation on the reals. Examples of such cardinal characteristics  are the Cicho\'n-characteristics as well as $\sfrak$ and $\rfrak$.
\smallskip

\noindent\textbf{\S\ref{sec:splpres}} We introduce the forcing $\Gor_\Bbf$, which has as support a suitable $2$-graph $\Bbf$.
We look at FS iterations of ccc posets, in general, whose initial part is a FS product of posets of the form $\Gor_\Bbf$ where $\Bbf$ is
in the ground model.
We define \emph{$\lambda$-small history iterations} (where on a dense set,
conditions have ${<}\lambda$-sized
\emph{history}), as well as \emph{symmetric} iterations, and show
that symmetric $\lambda$-small history iterations allow us to control $\sfrak$
(and later also $\rfrak$).
\smallskip

\noindent\textbf{\S\ref{sec:main1}} We define \emph{Suslin $\lambda$-small} iterations, which are $\lambda$-small history iterations, and give consequences of this notions, as well as sufficient conditions
to get symmetric ones.
\smallskip

\noindent\textbf{\S\ref{sec:left}}
Closely following~\cite{GKS}, we construct a symmetric Suslin-$\lambda$-small iteration $\Por^0$ that separates the cardinals on the left hand side of the diagram, with $\covM=\cfrak$ and $\sfrak=\pfrak$.\smallskip

\noindent\textbf{\S\ref{sec:15}} We show how the tools of~\cite{GKMS2,GKMS1}
can applied to $\Por^0$, resulting
in a forcing that gives the main theorem.\smallskip

\noindent\textbf{\S\ref{sec:disc}} We discuss some open questions related to this work.

\section{Suitable 2-graphs}\label{sec:graph}

In this section we
define and construct suitable $2$-graphs.

\begin{definition}\label{DefSbG}
   Say that $\Bbf:=\la B,R_0,R_1\ra$ is a \emph{$2$-edge-labeled graph}, abbreviated \emph{$2$-graph} or \emph{bi-graph}, if
   \begin{enumerate}[(i)]
       \item $R_0$ and $R_1$ are irreflexive symmetric relations on $B$,
       \item $R_0\cap R_1=\emptyset$.
   \end{enumerate}
   In other words: Between two nodes $x$ and $y$ there is at most one edge, with color $0$ or $1$. 

   Concerning $2$-graphs, we define the following notions.
   \begin{enumerate}[(1)]
       \item If $A\subseteq B$, denote $\Bbf|_A:=\la A,R_0|_A,R_1|_A\ra$ where $R_e|_A:=R_e\cap(A\times A)$.

       \item A partial function (or \emph{coloring}) $\eta$ from $B$ into $2$ \emph{respects $\Bbf$} if $\{\eta(a),\eta(b)\}\neq\{e\}$ whenever $e\in2$, $a,b\in \dom\eta$ and $aR_e b$.
   \end{enumerate}

   The $2$-graph of Figure~\ref{fig:2graphcolor} does not have a coloring (with full domain) respecting it.


   For an infinite cardinal $\lambda$, we say that $\Bbf$ is a \emph{suitable bi-graph of size $\lambda$} if   
   \begin{enumerate}[(i)]
   \setcounter{enumi}{2}
       \item for each $e\in\{0,1\}$, there is a $W_e\subseteq B$ such   that $W_e$ is a complete $R_e$-graph (as a subgraph of $\la B,R_e\ra$),
        \item $W_0\cap W_1 =\emptyset$ and $|B|=|W_0|=|W_1|=\lambda$,
        \item for any $a\in B$ and $e\in\{0,1\}$, there is a coloring $\eta\colon B\to \{0,1\}$ respecting $\Bbf$ such that $\eta(a) =e$, and
        \item for any $a,b\in B$, there is some automorphism $f$ on $\Bbf$ sending $a$ to $b$.
   \end{enumerate}
   Any suitable $2$-graph of size $\aleph_1$ is just called \emph{suitable $2$-graph (S2G)}.
\end{definition}



\begin{figure}
  {\fontsize{11pt}{14pt}\selectfont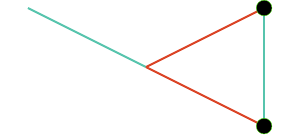}
  \caption{A finite $2$-graph which cannot be respected by any coloring.}
  \label{fig:2graphcolor}
\end{figure}

Properties (iii), (iv) and (vi) imply for all $b\in B$ and $e\in\{0,1\}$:
\begin{equation}\label{eq:S2G}
    \text{$b$ is contained in an uncountable $R_e$-complete subgraph of $\Bbf$.}
\end{equation}

In the rest of this section we will show that suitable graphs of any infinite size exist.



\begin{remark}

     In our applications, we only need the following weakening of property (v): for any $t\in[B]^{<\aleph_0}$, $a\in t$ and $e\in 2$, there is some $\eta :t\to 2$ that respects $\Bbf$ such that $\eta(a)=e$. The only place where (the weakening of) (v) is used is in the proof of Lemma~\ref{GBprop1}(b). However, this weakening is equivalent to (v) itself by compactness.
\end{remark}

\begin{definition}\label{f11}
    Let $\Bbf=\la B,R_0,R_1\ra$ be a bi-graph.
    \begin{enumerate}[(1)]
        \item A \emph{path in $\Bbf$} is a finite sequence $\seq{b_j}{j\leq n}$ of elements of $B$ such that, for $j<n$, there is some $e\in\{0,1\}$ such that $b_j \mathrel{R}_e b_{j+1}$. For $a,b\in B$, a \emph{path from $a$ to $b$ in $\Bbf$} is a path in $\Bbf$ as above such that $b_0=a$ and $b_n=b$.
        \item If $\Bbf' = \la B',R'_0,R'_1\ra$ is another bi-graph, then $\Bbf\preceq^-\Bbf'$ means that $\Bbf$ is a substructure of $\Bbf'$ and there is no path $\seq{b_j}{j\leq n}$ in $\Bbf'$ with $n\geq 2$ such that $\set{b_j}{j\leq n}\cap B = \{b_0,b_n\}$ and $b_0\neq b_n$.
    \end{enumerate}
\end{definition}

\begin{lemma}\label{f20}
    The relation $\preceq^-$ is a partial order in the class of bi-graphs. 
    Even more, if $\seq{\Bbf^\alpha = \la B^\alpha,R^\alpha_0,R^\alpha_1\ra}{\alpha<\gamma}$ is a $\preceq^-$-increasing sequence, then 
    \[\Bbf^\gamma := \left\la \bigcup_{\alpha <\gamma}B^\alpha,\bigcup_{\alpha <\gamma}R^\alpha_0, \bigcup_{\alpha <\gamma}R^\alpha_1 \right\ra\] 
    is the $\preceq^-$-supremum of $\set{\Bbf^\alpha}{\alpha<\gamma}$.
\end{lemma}
\begin{proof}
    We show that $\preceq^-$ is transitive (the rest is straightforward). So assume that $\Bbf^0\preceq^- \Bbf^1$ and $\Bbf^1\preceq^- \Bbf^2$ and, towards a contradiction, that there is a path $\seq{b_j}{j\leq n}$ in $\Bbf^2$ with $n\geq 2$ such that $\set{b_j}{j\leq n}\cap B^0 = \{b_0,b_n\}$ and $b_0\neq b_n$. This path may  contain cycles of the form $\seq{b_j}{j\in[j_0,j_1]}$ for some $j_0<j_1$ such that $\set{b_j}{j\in[j_0,j_1]}\cap B_1 = \{b_{j_0},b_{j_1}\}=\{b_{j_0}\}$ (i.e\ $b_{j_0}= b_{j_1}$). By removing such loops (and leaving the points in $B_1$), we may assume that the path $\seq{b_j}{j\leq n}$ does not contain such loops. 
    
    Since $\Bbf^0\preceq^- \Bbf^1$, we must have that the path $\seq{b_j}{j\leq n}$ intersects $B^2\smallsetminus B^1$, so we can find $j'_0<j'_1$ such that $\set{b_j}{j\in[j'_0,j'_1]}\cap B_1 = \{b_{j'_0},b_{j'_1}\}$ and $b_{j'_0}\neq b_{j'_1}$, which contradicts $\Bbf^1\preceq^- \Bbf^2$.
\end{proof}

\begin{theorem}\label{SbGexists}
    For any infinite cardinal $\lambda$, there is a suitable bi-graph of size $\lambda$.
\end{theorem}
\begin{proof}
    Our construction of the graph relies on the partial order $P$ whose elements are tuples $p=\la B^p,R^p_0, R^p_1,C^p,\bar f^p,\bar \eta^p\ra$ satisfying:
    \begin{enumerate}[({C}1)]
        \item $\Bbf^p := \la B^p,R^p_0,R^p_1\ra$ is a bi-graph with $B^p\subseteq \lambda^+$ of size ${\leq}\lambda$;
        \item $C^p\subseteq B^p\times B^p$;
        \item $\bar f^p = \seq{f^p_{a,b}}{(a,b)\in C^p}$ where each $f^p_{a,b}\colon D^p_{a,b}\to E^p_{a,b}$ is an isomorphism of substructures of $\Bbf^p$ such that $a\in D^p_{a,b}$, $\Bbf^p|_{D^p_{a,b}}\preceq^- \Bbf^p$, $b\in E^p_{a,b}$, $\Bbf^p|_{E^p_{a,b}}\preceq^- \Bbf^p$ and $f^p_{a,b}(a) = b$; and
        \item $\bar \eta^p = \seq{\eta^p_{a,e}}{a\in B^p,\ e\in\{0,1\}}$ such that each $\eta^p_{a,e}\colon B^p\to\{0,1\}$ is a coloring respecting $\Bbf^p$ sending $a$ to $e$;
    \end{enumerate}
    The order is defined by $q\leq p$ if
    \begin{enumerate}[({O}1)]
        \item $\Bbf^p\preceq^- \Bbf^q$,
        \item $C^p\subseteq C^q$,
        \item for each $(a,b)\in C^p$, $f^q_{a,b}$ extends $f^p_{a,b}$, and
        \item for $a\in B^p$ and $e\in\{0,1\}$, $\eta^q_{a,e}$ extends $\eta^p_{a,e}$.
    \end{enumerate}

    By recursion, we construct a decreasing sequence $\seq{p^\alpha}{\alpha\leq\lambda}$ in $P$, $p^\alpha:=\la B^\alpha,R^\alpha_0, R^\alpha_1,C^\alpha,\bar f^\alpha,\bar \eta^\alpha\ra$. Our desired graph will be $\Bbf^\lambda$.

    For the construction, we use a book-keeping bijection $h\colon \lambda \to 2\times {}^{6}\lambda$ such that, whenever $h(\alpha) = (i,t)$, $t_k\leq \alpha$ for all $k<6$. At each step $\alpha<\lambda$, we enumerate $B^\alpha = \set{b_{\alpha,\eta}}{\eta<\lambda}$.

    In the base step $\alpha=0$, pick disjoint subsets $W_0$ and $W_1$ of $\lambda^+$ of size $\lambda$ and define $\Bbf^0:=\la B^0,R^0_0,R^0_1\ra$ such that, for $e\in\{0,1\}$ and $a,b\in B^0$:
    \begin{itemize}
        \item $B^0 := W_0\cup W_1$,
        \item $a \mathrel{R}^0_e b$ iff $a\neq b$ are in $W_e$,
    \end{itemize}
    We set $C^0 := \emptyset$. For each $a\in B^0$ and $e\in\{0,1\}$, it is easy to find a desired coloring $\eta^0_{a,e}\colon B^0\to \{0,1\}$. Concretely,
    \[
     \eta^0_{a,e}(x) := \left\{
        \begin{array}{ll}
           e  &  x\in W_{1-e}\cup\{a\},\\
           1-e  & x\in W_e\smallsetminus \{a\}.
        \end{array}
     \right.    
    \]
    
     In the successor step $\alpha+1$, let $h(\alpha) = (i,t)$. We proceed by cases according to the value $i$.     

    Case $i=0$: Let $a_0:= b_{t_0,t_1}$, $a_1:= b_{t_2,t_3}$ and $c:= b_{t_4,t_5}$, which are in $B^\alpha$. We aim to construct a $p^{\alpha+1} \leq p^\alpha$ such that $(a_0,a_1)\in C^{\alpha+1}$ and $c\in D^{\alpha+1}_{a_0,a_1}$. 
    
    Define $p :=\la B,R_0,R_1,C,\bar f,\bar \eta\ra\in P$ as $p^\alpha$ but with $C:= C^\alpha\cup\{(a_0,a_1)\}$ and, in the case $(a_0,a_1)\notin C^\alpha$, let $f^p_{a_0,a_1}$ be the only function from $\{a_0\}$ onto $\{a_1\}$ (notice that singletons are $\preceq^-$-below a bi-graph). Clearly $p\in P$ and $p\leq p^\alpha$. In the case $c\in D^p_{a_0,a_1}$, let $p^{\alpha+1}:=p$. So assume that $c\notin D:= D^p_{a_0,a_1}$. 
    We say that $\seq{x_\ell}{\ell\leq n}$ is a \emph{path from $c$ to $D$} if it is a path in $\Bbf$ such that $x_0 = c$, $x_n\in D$ and $x_\ell \notin D$ when $0< \ell <n$. The point $x_n$ is called the \emph{end-point of the path}. Since $\Bbf|_D\preceq^- \Bbf$, any two paths from $c$ to $D$ must go to the same end-point, which we denote by $z^*$ if it exists (i.e.\ if a path from $c$ to $D$ exists). Let $D':= \{c\} \cup \set{x_\ell\in B\smallsetminus\{z^*\}}{\bar x \text{ is a path from }c \text{ to }D}$ and $D^*:=D\cup D'$. Let $E'\subseteq \lambda^+$ be a copy of $D'$ disjoint with $B$, $E^*:= E\cup E'$ where $E:= E^\alpha_{a_0,a_1}$, and let $f^*\colon D^*\to E^*$ be a bijection extending $f:=f^p_{a_0,a_1}$. Let $B^{\alpha+1}:= B\cup E'$ and let $\Bbf^{\alpha+1}$ be the smallest bi-graph expanding $\Bbf$ such that $f^*$ is an isomorphism of substructures. Also let $C^{\alpha+1}:= C$, $f^{\alpha+1}_{a,b}:= f^p_{a,b}$ for any $(a,b)\in C\smallsetminus \{(a_0,a_1)\}$, and $f^{\alpha+1}_{a_0,a_1}:= f^*$. 
    
    We have that $\Bbf|_{D^*}\preceq^- \Bbf\preceq^- \Bbf^{\alpha+1}$ and $\Bbf^{\alpha+1}|_{E^*}\preceq^- \Bbf^{\alpha+1}$. 
    To see $\Bbf|_{D^*}\preceq^- \Bbf$: if $\bar y=\seq{y_j}{j\leq n}$ is a path in $\Bbf$ from $y_0\in D^*$ to $y_n\in D^*$ and $y_0\neq y_n$ then:
    in the case $y_0,y_n\in D$, $y_j\in D$ for some $0<j<n$ because $\Bbf|_D\preceq^- \Bbf$; otherwise, the path can be extended (on both sides) to a path from $c$ to $D$, so it lies inside $D^*$.

    Observe that, whenever $x\in B$, $y\in E' = B^{\alpha+1}\smallsetminus B$ and $x \mathrel{R}_e y$ for some $e$, $x=f(z^*)$. This implies $\Bbf\preceq^- \Bbf^{\alpha+1}$ and $\Bbf^{\alpha+1}|_{E^*}\preceq^- \Bbf^{\alpha+1}$ (also because $f(z^*) \in E$ and $\Bbf|_E\preceq^- \Bbf$).

    We now define the colorings $\eta^{\alpha+1}_{a,e}$ for $a\in B^{\alpha+1}$ and $e\in\{0,1\}$. When $a\in B$, extend $\eta_{a,e}$ by defining $\eta^{\alpha+1}_{a,e}(f^*(x)):=\eta_{z^*,\eta_{a,e}(f(z^*))}(x)$ for $x\in D'$ in case $z^*$ exists (otherwise $E' = \{f^*(c)\}$ and $f^*(c)$ does not have neighbors, so it can be colored arbitrarily). When $a = f^*(z)$ for some $z\in D'$, define
    \[\eta^{\alpha+1}_{a,e}(y) := \left\{
      \begin{array}{ll}
          \eta_{z,e}(x) & \text{$y=f(x)$ for some $x\in D'$,} \\
          \eta_{f(z^*),\eta_{z,e}(z^*)}(y) &  y\in B
      \end{array}
    \right.\]
    in case $z^*$ exists (otherwise $z=c$, so $\eta_{f(z^*),\eta_{z,e}(z^*)}$ can be replaced by anything in $\bar \eta^\alpha$).
    All the above clearly defines a $p^{\alpha+1} \leq p$.

    Case $i=1$: Let $a_0:= b_{t_0,t_1}$, $a_1:= b_{t_2,t_3}$ and $d:= b_{t_4,t_5}$, which are in $B^\alpha$. Proceed similarly as in the previous proof to construct $p^{\alpha+1} \leq p^{\alpha}$ such that $(a_0,a_1)\in C^{\alpha+1}$ and $d\in E^{\alpha+1}_{a_0,a_1}$.

    In the limit step $\gamma\leq\lambda$, we let $p^\gamma$ be the infimum of $\set{p^\alpha}{\alpha<\gamma}$ in $P$. We show that this infimum exists (which shows that any decreasing sequence in $P$ of length ${<}\lambda^+$ has an infimum, i.e.\ $P$ is ${<}\lambda^+$-closed). Let $\Bbf^\gamma$ be the $\preceq^-$-supremum of $\set{\Bbf^\alpha}{\alpha<\gamma}$, $C^\gamma:=\bigcup_{\alpha<\gamma}C^\alpha$, $f^\gamma_{a,b}:= \bigcup_{\alpha<\gamma} f^{\alpha_{a,b}}$ and $\eta^\gamma_{a,e}:=\bigcup_{\alpha<\gamma} \eta^\alpha_{\alpha,e}$ (abusing notation in the last two unions, since formally the union starts at some $\alpha_0$ from where the functions start to appear), which defines this infimum $p^\gamma$.

    At the end, our book-keeping functions and the construction ensures that $C^\lambda = B^\lambda\times B^\lambda$ and $\dom f^\lambda_{a,b} = \ran f^\lambda_{a,b} = B^\lambda$ for all $(a,b)\in C^\lambda$. Hence, $\Bbf^\lambda$ is as desired.
\end{proof}

\begin{remark}
    In the previous proof, we could have argued a bit differently, using that the partial order $P$ is ${<}\lambda^+$-closed. In the successor step of the proof, we have actually proved that the following sets are dense in $P$ for $a,b,c\in \lambda^+$:
    \begin{align*}
        A^0_{a,b,c} & := \set{p\in P}{(a,b)\in C^p,\ c\in D^p_{a,b}},\\
        A^1_{a,b,c} & := \set{p\in P}{(a,b)\in C^p,\ c\in E^p_{a,b}}.
    \end{align*}
    Since $P$ is ${<}\lambda^+$-closed, we can find a filter $G\subseteq P$ with $p^0\in G$ intersecting all the dense sets above. By taking unions of the components of the members of the filter (similar to the limit step of the proof above), we obtain a bi-graph $\Bbf^+ = \la \lambda^+,R_0,R_1\ra$ satisfying the conditions of a suitable bi-graph of size $\lambda^+$, with the exception that $W_0$ and $W_1$ have size $\lambda$. Now, pick a large enough regular cardinal $\chi$ and an elementary submodel $N\preceq H_\chi$ of size $\lambda$ with $\lambda\cup\{\Bbf^+\}\subseteq N$. Then, $\Bbf:=\Bbf^+|_{\lambda^+\cap N}$ is a suitable bi-graph of size $\lambda$.
\end{remark}

\section{Cardinal Characteristics, COB and LCU}\label{sec:COB}

Many classical characteristics can be defined by the framework of relational systems
as in e.g.~\cite{MR1234291,Blass}.
Say that $\Rbf:=\la X,Y,R\ra$ is a \emph{relational system} if $X$ and $Y$ are non-empty sets,
and $R$ is a relation.
The following cardinal characteristics are associated with $\Rbf$.\smallskip

$\dfrak(\Rbf):=\min\{|D|: D\subseteq Y\text{\ and }\forall x\in X\,\exists y\in D\,(x R y)\}$;\smallskip

$\bfrak(\Rbf):=\min\{|F|: F\subseteq X\text{\ and }\neg\exists y\in Y\,\forall x\in X\,(x R y) \}$.\medskip

In this work, we are particularly interested in relational systems $\Rbf$ such that
\begin{enumerate}[({RS}1)]
  \item $X$ and $Y$ are subsets of Polish spaces $Z_0$ and $Z_1$, respectively, and absolute for transitive models of ZFC (e.g. they are analytic);
  \item $R\subseteq Z_0\times Z_1$ is absolute for transitive models of ZFC (e.g.\ analytic in $Z_0\times Z_1$).
\end{enumerate}
When these properties hold we say that $\Rbf$ is a \emph{relational system of the reals}.
In all the cases explicitly mentioned throughout this paper, $X$ and $Y$ are Polish spaces themselves and $R$ is Borel in $X\times Y$. In this case, there is no problem to identify $X=Y=\omega^\omega$, and we call $\Rbf$,
or rather the characteristics $\bfrak(\Rbf)$
and $\dfrak(\Rbf)$, \emph{Blass-uniform} (cf.~\cite[\S2]{GKMS1}).

\begin{example}\label{exp:blassunif}(\cite[2.2.2]{MR1234291} or \cite[\S4 \& \S5]{Blass})
The splitting number $\sfrak$ and the reaping number $\rfrak$ are  Blass-uniform: Denote
$\Rbsp:=\la2^\omega,[\omega]^{\aleph_0},\Rsp\ra$ where $x\Rsp y$ iff $x{\upharpoonright}y$ is constant except in finitely many points of $y$. Then $\sfrak=\bfrak(\Rbsp)$ and $\rfrak=\dfrak(\Rbsp)$.\footnote{It would be more natural to consider the relational system $\la[\omega]^{\aleph_0},[\omega]^{\aleph_0},R\ra$ where $xRy$ iff either $x\supseteq^* y$ or $\omega\smallsetminus x\supseteq^* y$, but $\Rbsp$ is more suitable in our proofs. It is not hard to see that both relational systems are Tukey-equivalent.}

Also all Cicho\'n-characteristics are Blass-uniform. The Blass-uniform relational systems we use for these characteristics are (as in the Cicho\'n's Maximum constructions) in some instances slightly different from the ``canonical" ones. See e.g.~\cite[Ex.~2.16]{diegoetal},~\cite[Ex.~2.10]{modKST} and~\cite[\S1]{GKS} for the definition of the Blass-uniform relational systems corresponding to the Cicho\'n-characteristics.
\end{example}

As in \cite{GKMS2} we also look at relational systems $S=\la S,S,\leq\ra$ where $\leq$ is an upwards directed partial order on $S$. Here
$\cp(S):=\bfrak(S)$ is the \emph{completeness of $S$}, and $\cf(S):=\dfrak(S)$ is the \emph{cofinality of $S$}. Recall that, whenever $S$ has no greatest element, $\cp(S)\leq\cf(S)$, and equality holds when the order is linear.

The following is a very useful notion to calculate the value of cardinal characteristics (specially in forcing extensions).

\begin{definition}[cf.~{\cite[\S1]{GKS}}]\label{def:COB}
   Fix a directed partial order $S=\la S,\leq\ra$ and a relational system $\Rbf=\la X,Y,R\ra$. Define the property:\medskip

   \noindent\textbf{Cone of bounds.}\\
   $\COB_\Rbf(S)$ means:
   There is a family $\bar{y}=\{y_i:i\in S\}\subseteq Y$ such that
   \[\forall x\in X\, \exists i_x\in S\, \forall j\geq i_x\, (x Ry_j).\]

   When $L=\la L,\leq\ra$ is a linear order, we additionally define

   \noindent\textbf{Linear cofinally unbounded.}\\
   $\LCU_\Rbf(L)$ means:
   There is a family $\bar{x}=\{x_i : i\in L\}\subseteq X$ such that
   \[\forall y\in Y\,\exists i\in L\,\forall j\geq i\,(\neg(x_j R y)).\]
\end{definition}

In the following remarks we address very natural characterizations and consequences of these properties.

\begin{remark}[Tukey connections and $\COB$]\label{rem:COB}
   Let $\bar{y}$ be a witness of $\COB_\Rbf(S)$. By the definition of $\COB_\Rbf(S)$ we have that the functions $f:X\to S$ and $g:S\to Y$, defined by $f(x):=i_x$ and $g(i):=y_i$, form a Tukey connection from $\Rbf$ into $S$. So we conclude that
   \[\COB_\Rbf(S)\text{\ holds iff }\Rbf\leqT S,\]
   where $\leqT$ denotes the Tukey order.
\end{remark}

\begin{remark}[Duality and $\LCU$]\label{rem:LCU}
   Let $\Rbf=\la X,Y,R\ra$ be a relational system. The \emph{dual of $\Rbf$} is the relational system $\Rbf^\perp:=\la Y,X,R^\perp\ra$ where $uR^\perp v \Leftrightarrow \neg(v R u)$. It is clear that $\dfrak(\Rbf^\perp)=\bfrak(\Rbf)$ and $\bfrak(\Rbf^\perp)=\dfrak(\Rbf)$. Also, given a linear order $L$,
   \[\LCU_\Rbf(L)\text{\ iff }\COB_{\Rbf^\perp}(L).\]
   Hence, by Remark~\ref{rem:COB},
   \[\LCU_\Rbf(L)\text{\ iff }\Rbf^\perp\leqT L.\]
   When $L$ has no greatest element, $L^\perp$ is Tukey-equivalent to $L$, so
   \[\LCU_\Rbf(L)\text{\ iff }L\leqT \Rbf.\]
   Although $\LCU$ is a particular case of $\COB$, they are used with different roles in our applications, so it is more practical to use different notations.
\end{remark}

As a direct consequence of these remarks:

\begin{lemma}[cf.~{\cite[\S1]{GKS}}]\label{lem:COBbounds}
   Let $\Rbf$ be a relational system, $S$ a directed partial order and let $L$ be a linear order without greatest element. Then
   \begin{enumerate}[(a)]
     \item $\COB_\Rbf(S)$ implies $\cp(S)\leq\bfrak(\Rbf)$ and $\dfrak(\Rbf)\leq\cf(S)$.
     \item $\LCU_\Rbf(L)$ implies $\bfrak(\Rbf)\leq\cp(L)=\cf(L)\leq\dfrak(\Rbf)$.
   \end{enumerate}
\end{lemma}

In our applications we aim to force $\COB_\Rbf(S)$ and $\LCU_\Rbf(L)$ for a given relational system of the reals $\Rbf$;  this will help us compute the value of $\bfrak(\Rbf)$ and $\dfrak(\Rbf)$ in generic extensions. For this purpose, the following variation of Definition~\ref{def:COB} is very practical.

\begin{definition}[{\cite{GKMS2}}]\label{def:COBforcing}
   Let $\Rbf=\la X,Y,R\ra$ be a relational system of the reals, $S=\la S,\leq_S\ra$ a directed partial order, $L=\la L,\leq_L\ra$ a linear order, and let $\Por$ be a forcing notion. Define the following properties.\medskip

   \noindent$\COB_\Rbf(\Por,S)$: There is a family $\dot{\bar{y}}=\{\dot{y}_i:i\in S\}$ of $\Por$-names of members of $Y^{V^\Por}$ such that, for any $\Por$-name $\dot{x}$ of a member of $X^{V^\Por}$ there is some $i\in S$ such that
   \[\Vdash_\Por \forall j\geq_S i\, (\dot{x}R\dot{y}_j).\]

   \noindent$\LCU_\Rbf(\Por,L)$: There is a family $\dot{\bar{x}}=\{\dot{x}_i:i\in L\}$ of $\Por$-names of members of $X^{V^\Por}$ such that, for any $\Por$-name $\dot{y}$ of a member of $Y^{V^\Por}$ there is some $i\in L$ such that
   \[\Vdash_\Por \forall j\geq_L i\, (\neg(\dot{x}_jR\dot{y})).\]
\end{definition}

\begin{remark}\label{rem:COB2}
   Concerning the properties $\COB_\Rbf(\Por,S)$ and $\LCU_\Rbf(\Por,L)$, the relational system $\Rbf$ (i.e., both base sets as well as the relation) are interpreted in the generic extension (this is why we required these objects to be definable), while $S$ and $L$ are taken as sets in the ground model (not interpreted).

   It is clear that $\COB_\Rbf(\Por,S)$ implies $\Vdash_\Por\COB_\Rbf(S)$. Although the converse is not true in general, it holds in the cases we are interested in, when $\Por$ is ccc and $\cp(S)$ is uncountable. More precisely, if $\cp(S)$ is uncountable and $\Por$ is $\cp(S)$-cc then $\COB_\Rbf(\Por,S)$ is equivalent to $\Vdash_\Por\COB_\Rbf(S)$. Moreover, $\Por$ forces $\cp(S)^{V^\Por}=\cp(S)^V$ and $\cf(S)^{V^\Por}=\cf(S)^V$,
   so, by Lemma~\ref{lem:COBbounds}, in the generic extension $\COB_\Rbf(S)$ implies $\cp(S)^V\leq\bfrak(\Rbf)$ and $\dfrak(\Rbf)\leq\cf(S)^V$.

   Likewise, $\LCU_\Rbf(\Por,L)$ implies $\Vdash_\Por\LCU_\Rbf(L)$, and the converse holds whenever $L$ has no greatest element, $\cf(L)$ is uncountable and $\Por$ is $\cf(L)$-cc.

   However, the restriction ``$\cp(S)$ is uncountable and $\Por$ is $\cp(S)$-cc" is not required for the following result.
\end{remark}

\begin{lemma}[{\cite[Lemma~1.3]{GKMS2}}]\label{lem:COBforbd}
   Let $\Rbf$ be a relational system of the reals, $S$ a directed partial order without greatest element, and let $\Por$ be a forcing notion. If $\mu=\cp(S)^V$ and $\lambda=\cf(S)^V$, then
   \begin{enumerate}[(a)]
     \item $\COB_\Rbf(\Por,S)$ implies $\Vdash_\Por$``$\mu\leq\bfrak(\Rbf)$ and $\dfrak(\Rbf)\leq|\lambda|$".
     \item If $L=S$ is a linear order, then $\LCU_\Rbf(\Por,L)$ implies
      \[\Vdash_\Por\text{``}\bfrak(\Rbf)\leq|\lambda|\leq\lambda\leq\dfrak(\Rbf)\text{"}.\]
   \end{enumerate}
\end{lemma}

\section{Preserving splitting families with symmetric iterations}\label{sec:splpres}

\subsection{The single forcing \texorpdfstring{$\Gor_\Bbf$}{GB}}

Using suitable $2$-graphs, we define
a poset which will be used as factor
for the
forcing adding
the splitting families we aim to preserve.

\begin{definition}\label{DefGB}
   Let $\Bbf=\la B,R_0,R_1\ra$ be a suitable $2$-graph. Define the forcing $\Gor_\Bbf$ whose conditions are functions $p:F_p\times n_p\to\{0,1\}$ where $F_p\in[B]^{<\aleph_0}$ and $n_p<\omega$ (also demand $F_p=\emptyset$ iff $n_p=\emptyset$). The order is defined by $q\leq p$ iff
   \begin{enumerate}[(i)]
       \item $p\subseteq q$,
       \item for each $k\in[n_p,n_q)$, the map $F_p\to 2$, $a\mapsto q(a,k)$ respects $\Bbf$, that is,
       if $e\in\{0,1\}$, $a,b\in F_p$, and $a R_e b$, then $\{q(a,k),q(b,k)\}\neq\{e\}$.
   \end{enumerate}
   For $a\in B$ denote by $\dot{\eta}_a$ the name of the generic real added at $a$, that is, $\Gor_\Bbf$ forces that, for any $k<\omega$, $\dot{\eta}_a(k)=e$ iff $p(a,k)=e$ for some $p$ in the generic set.

   For $p\in\Gor_\Bbf$ denote $\supp p:=F_p$.
\end{definition}


\begin{lemma}\label{GBprop1}
   Let $\Bbf=\la B,R_0,R_1\ra$ be a suitable $2$-graph. Then:
   \begin{enumerate}[(a)]
       \item $\Gor_\Bbf$ is $\sigma$-centered.
       \item For any $a\in B$, $\Gor_\Bbf$ forces that $\dot{\eta}_a$ is Cohen over $V$.
       \item Any $p\in\Gor_\Bbf$ forces that, for any $k\geq n_p$, the map $F_p\to 2$, $a\mapsto \dot{\eta}_a(k)$ respects $\Bbf$, that is, if $e\in\{0,1\}$, $a,b\in F_p$ and $a R_e b$, then  $\dot{\eta}_a(k)$ and $\dot{\eta}_b(k)$ cannot both be $e$ at the same time.
       \item\label{item:dafterall}
         Assume for $i\in\{1,2\}$:
         \begin{itemize}
             \item $e\in\{0,1\}$, $p_i\in\Gor_\Bbf$, $c_i\in F_{p_i}$, $c_1 R_e c_2$,
             \item $\Qor$ is a poset, $\Gor_\Bbf\lessdot\Qor$,
             \item $\dot b$ is a $\Qor$-name of an infinite subset of $\omega$,
             \item $q_i\leq p_i$ in $\Qor$ and $q_i\Vdash_\Qor\dot\eta_{c_i}{\restriction}\dot b\equiv e$,
         \end{itemize}
         Then $q_1$ and $q_2$ are incompatible.


       \item\label{item:notd} If $f:B\to B$ is a $\Bbf$-automorphism,
       then $\hat f:\Gor_\Bbf\to \Gor_\Bbf$
       defined by $ \hat f(p)(\alpha,n) = p(f^{-1}(\alpha), n)$ (where $F_{\hat{f}(p)}:=f[F_p]$),
       is a p.o.-automorphism.
   \end{enumerate}
\end{lemma}
\begin{proof}
   To see (a), first note that since $|B\times\omega|=\aleph_1$, by Engelking--Kar{\l}owicz~\cite{EngKarl} there is a countable set $H\subseteq 2^{B\times\omega}$ such that any finite partial function from $B\times\omega$ into $2$ can be extended by some member of $H$. For $h\in H$ and $n<\omega$, let $C_{h,n}:=\{p\in\Gor_\Bbf:p\subseteq h\text{\ and }n_p=n\}$. It is clear that $C_{h,n}$ is centered and $\Gor_\Bbf=\bigcup_{h\in H}\bigcup_{n<\omega}C_{h,n}$, so $\Gor_\Bbf$ is $\sigma$-centered.\smallskip

  \noindent (b): Consider Cohen forcing $\Cor:=2^{<\omega}$ ordered by end-extension. For $a\in B$ define $\proj_a:\Gor_\Bbf\to\Cor$ such that, for any $p\in\Gor_\Bbf$, $\proj_a(p):=\la p(a,k) : k<n_p\ra$ if $a\in\supp p$, or $\proj_a(p)$ is the empty sequence otherwise. It is enough to show that $\proj_a$ is a forcing projection, that is,
   \begin{enumerate}[(i)]
      \item for any $p,q\in\Gor_\Bbf$ if $q\leq p$ then $\proj_a(q)\supseteq\proj_a(p)$,
      \item for any $p\in\Gor_\Bbf$ and $s\in\Cor$, if $s\supseteq\proj_a(p)$ then there is some $q\leq p$ in $\Gor_B$ such that $\proj_a(q)\supseteq s$ (even $\proj_a(q)=s$),
      \item $\proj_a[\Cor_\Bbf]$ is dense in $\Cor$ (even $\proj_a$ is onto).
   \end{enumerate}
   Property (i) is easy, (ii) follows by Definition~\ref{DefSbG}(v), and (iii) follows by (ii) and the fact that $\proj_a(\emptyset)=\la\ \ra$.\smallskip

   \noindent (c): By the definition of the order of $\Gor_\Bbf$.\smallskip

   \noindent (\ref{item:dafterall}): Assume $q\in\Qor$ is stronger than $q_1$ and $q_2$, so $q\Vdash$``$\{k<\omega:\, \dot\eta_{c_1}(k)=\dot\eta_{c_2}(k)=e\}$ is infinite". Hence, there is some $p\in\Gor_\Bbf$ stronger than $p_1$ and $p_2$ forcing the same, but this contradicts (c) because $c_1,c_2\in F_p$ and $c_1 R_e c_2$.\smallskip

   \noindent (\ref{item:notd}) is straightforward.
\end{proof}

\begin{remark}\label{rem:blalb}
  The obvious restriction
  of $\Gor_\Bbf$ to, say, the first two coordinates, is not a projection, and $\Gor_\Bbf$ is not a FS iteration
  of length $\omega_1$ in any natural way.
  Assume, e.g.,
  we restrict to
  $\{0,1\}\subseteq B=\omega_1$, and  $\Bbf$
  contains an $e$-colored edge from node $e$ to node $2$ for $e\in\{0,1\}$. Start with a condition $p: \{0,1,2\}\times n\to 2$ (for e.g.\ $n=1$),
  restrict it to $p^-=p{\restriction} \{0,1\}$ and extend it to $p'\in \Gor_{\Bbf{\restriction}\{0,1\}}$ by
  setting $p'(e,n)=e$
  for $e\in\{0,1\}$. Then there is no
  $q\in \Gor_\Bbf$, $q\le  p$,  compatible with $p'$.
\end{remark}

We will use FS iterations where the first step is given by a FS product of posets of the form $\Gor_\Bbf$ as above.
It is clear that, if $\Bbf$ is a S2G in the ground model, then it is still a S2G in any
extension preserving $\omega_1$. On the other hand, constructing $\Gor_\Bbf$ from $\Bbf$ is absolute for transitive models of ZFC, so any finite support product of posets of the form $\Gor_\Bbf$ is forcing equivalent to their finite support iteration (as long as the sequence of $2$-graphs lives in the ground model).

\subsection{Suitable iterations, nice names and automorphisms}

We now introduce some notions associated with these iterations, relevant for the preservation of splitting families.

 From this point on, products of  ordinals (such as $\omega_1\pi$)  should be interpreted as ordinal products.

\begin{definition}\label{DefSI}
    A \emph{suitable iteration} is defined by the
    following objects:
    \begin{enumerate}[(I)]
        \item A cardinal $\pi^\tbf_0>0$.
        \item For each $\delta<\pi^\tbf_0$, a S2G $\Bbf^\tbf_\delta=\la B^\tbf_{\delta},R^\tbf_{\delta,0},R^\tbf_{\delta,1}\ra$ with  $B^\tbf_\delta:=[\omega_1\delta,\omega_1(\delta+1))$,
        \item an ordinal $\pi^\tbf\geq\pi^\tbf_1:=\omega_1\pi^\tbf_0$,
        \item a FS ccc iteration $\Por^\tbf$ of length $1+(\pi^\tbf- \pi_1^\tbf)$
        where the first iterand is
        the FS product of the $\Gor_{\Bbf_\delta}$ for $\delta<\pi_0^\tbf$, called
        $\Por^\tbf_{\pi_1^\tbf}$,
        and the following iterands are indexed
        by $\xi\in \pi^\tbf\smallsetminus \pi_1^\tbf$ and are
        ccc posets called $\Qnm^\tbf_\xi$.
\end{enumerate}
         As usual, we denote with $\Por^\tbf_\xi$ the result of the iteration up to $\xi$ (for $\pi^\tbf_1\le \xi\le \pi^\tbf$),
         and use $\Por^\tbf$ to denote either
         $\Por^\tbf_{\pi^\tbf}$ or the whole
         iteration (or its definition). See Figure~\ref{fig:iteration} for an illustration.
\end{definition}


    \begin{remark}
    Note that we could also view $\Por_{\pi_1}$ as (the result of) a FS-iteration of length $\pi_0$ (instead of length $1$, as we do in the definition).
    Then we would get an iteration $\Por$ of
    $\pi_0+(\pi-\pi_1)$.
    However, $\Por_{\pi_1}$
     is not a FS iteration of length $\pi_1$, at least not with natural iterands, see Remark~\ref{rem:blalb}.
    \end{remark}

Let us mention some notation:
\begin{notation}
\begin{enumerate}[(1)]
    \item A real-number-poset is a poset whose universe is a subset of the set of real numbers. For simplicity, we identify  the ``set of real numbers'' with the power set of $\omega$.
    \item For notational simplicity we will often identify
        $\Por_{\zeta+1}$
    (a set of partial functions)
    with
        $\Por_\zeta*\Qnm_\zeta$
    (a set of pairs $(p,q)$ with $p\in \Por_\zeta$ and $p\Vdash q\in \Qnm_\zeta$).
    \item Similarly, we will not distinguish between sequence of names and names of sequences.
\end{enumerate}

\end{notation}

    We now define the ``support'' $\supp(p)\subseteq \pi$
    of a condition $p$ (as opposed to the domain $\dom(p)$,
    which is, as we are dealing with a FS iteration,
    a finite subset of the index set
    $\{0\}\cup (\pi\smallsetminus \pi_1)$). We will
    also define the ``history'' $H$
    of a name and of a condition:
    \begin{definition}\label{def:suppH}
    Let $\Por$ be a suitable iteration.

    \begin{enumerate}[(1)]
        \item\label{FSPsupp}
        For $p\in\Por_{\pi_1}$ set
        $\supp( p):=\bigcup_{\delta\in\dom p}\supp( p(\delta))\subseteq\pi_1$.
        For $p\in\Por$, set
        $\supp p:=\supp(p(0))\cup(\dom(p)\smallsetminus \{0\})$ (or just $\dom(p)$, if $0\notin\dom(p)$).%
        \footnote{Recall that according to our indexing,
        $\dom(p)$ is a finite subset
        of $\{0\}\cup (\pi\smallsetminus \pi_1)$
        (where we interpret a FS condition $p$ as a partial function from the index $\pi_0$ with finite domain $\dom(p)$).
        Recall that $q:=p(0)\in
        \Por_{\pi_1}$, which is the FS product of $\Gor_{\Bbf_\delta}$ for $\delta<\pi_0$.
        So $q$ has a finite domain $\dom(q)\subseteq \pi_0$, and if $\delta\in\dom(q)$,
        then $q(\delta)\in\Gor_{\Bbf_\delta}$, so $X_\delta=\supp (q(\delta))$ (in the sense of the forcing
        $\Gor_{\Bbf_\delta}$) is a finite subset of $[\omega_1\delta,\omega_1(\delta+1))$.
        According to our definition, $\supp(q)=\bigcup_{\delta\in \dom(q)} X_\delta$.}


        \item\label{hist} For  $p\in\Por$ and a $\Por$-name $\tau$, we define $H(p)\subseteq\pi$
        and $H(\tau)\subseteq\pi$ as follows:
    \begin{enumerate}[(i)]
        \item For $p\in\Por_{\pi_1}$, $H(p):=\supp p$.
    \end{enumerate}
        For
        $\xi\ge \pi_1$ we
        define $H$
        by recursion on $\xi$
        for $p\in \Por_\xi$ and  for a $\Por_\xi$-name
        $\tau$.
        (We assume that
        $H(r)$ has been defined for all $r\in\Por_\zeta$
        for $\pi_1\leq \zeta<\xi$
        and $H(\sigma)$
        for all $\Por_\zeta$-names
        for $\pi_1\le \zeta<\xi$):
    \begin{enumerate}[(i)]
        \setcounter{enumii}{1}

        \item For $\xi=\zeta+1$ and $p\in\Por_{\zeta+1}$,
           \[H(p):=\left\{\begin{array}{ll}
                H(p{\upharpoonright}\zeta) & \text{if $\zeta\notin\supp p$,}  \\
                H(p{\upharpoonright}\zeta)\cup\{\zeta\}\cup H(p(\zeta)) & \text{if $\zeta\in\supp p$.}
           \end{array}\right.\]
           (Here, $H(p(\zeta))$ is defined because $p(\zeta)$ is a $\Por_\zeta$-name.)
        \item When $\xi>\pi_1$ is limit and $p\in\Por_\xi$, then $H(p)$ has already been defined (because $p\in\Por_\zeta$ for some $\zeta<\xi$).
        \item
        For any $\Por_\xi$-name $\tau$ define (by $\in$-recursion on $\tau$) \[H(\tau):=\bigcup\{H(\sigma)\cup H(p):(\sigma,p)\in\tau\}.\]

    \end{enumerate}

    \end{enumerate}
\end{definition}
Note that $H(\check x)=\emptyset$
for any standard name $\check x$.\footnote{A
standard name $\check x = \{(\check y, \mathbbm 1): y\in x\}$ (for $x\in V$)
hereditarily only uses the weakest condition $\mathbbm 1$, which in our case (an iteration) is the empty partial function;
accordingly $H(\check x)=\emptyset$.
If the reader
prefers a different formal definition
of FS iteration, then they should modify
the definition of $H$ to make sure that
$H(\check x)=\emptyset$.}

\begin{figure}
    \includegraphics[width=\textwidth]{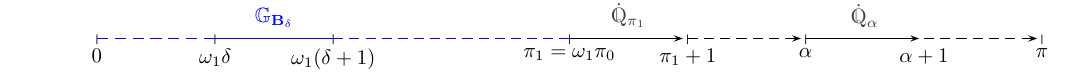}
    \caption{A suitable iteration. 
    $\pi_1=\omega_1\pi_0$ 
    is partitioned into $\pi_0$-many intervals of length $\omega_1$, and $B_\delta:=[\omega_1\delta,\omega(\delta+1))$, the set of vertices of the graph $\Bbf_\delta$, is the $\delta$-th interval of this partition.
    A suitable iteration is a FS product of the $\Gor_{\Bbf_\delta}$ for $\delta<\pi_0$, followed by a FS iteration of ccc posets. 
    The iterands of the FS iteration that follow are indexed by $\alpha\in[\pi_1,\pi)$.}
    \label{fig:iteration}
\end{figure}

\begin{remark*}
$H$ is not a ``robust'' notion:
$\Vdash\tau=\tau'$ does not imply $H(\tau)=H(\tau')$.
Still, it is a very natural and useful notion, which has appeared (in slightly different contexts) many times in forcing theory:
If $\tau$ is a $\Por_\pi$-name, then $H(\tau)\subseteq\pi$ is the set of coordinates the name $\tau$ ``depends on'', more concretely, $\tau$ can be calculated (by a function defined in $V$)
from the sequence of generic objects at the indices in $H(\tau)$.
\\
In the case of FS iterations where all iterands are real-number-posets (as in~\cite{ShCov,GKS}), $H(p)$ is countable for $p$ in a dense set;
and ``hereditarily nice names'' for reals will also have countable history.
In this paper we have to use hereditarily ${<}\lambda$-names (even for nice names of reals), the reason is indicated in Remark~\ref{rem:lambdarequired}.
\end{remark*}

Let us fix some notation regarding the
well-known ``nice names'':
\begin{definition}\label{Defnicenm}
Let $A$ and $B$ be subsets of $\Por$.
\begin{enumerate}[(1)]
    \item
   A $\Por$-name
   $\dot r$ is a \emph{nice name for a subset of $\omega$, determined by $A$}, if
   $\dot r$ has the form $\bigcup_{n\in \omega} \{(\check n,q): q \in A_n\}$, where each $A_n$ is a (possibly empty) antichain in $\Por$, and
   $A=\bigcup_{n\in\omega} A_n$.

    \item
    Analogously, $\dot Q$ is a \emph{nice name for a real-number-poset
    of size ${<}\lambda$ ,
    determined by $B$},
    if there is a $\mu<\lambda$
    such that $\dot Q$ is a sequence
    $\la \dot{r}_i\ra_{i\in\mu}$ of nice names for subsets of $\omega$ determined by $A_i$,
    together with a sequence $\la \dot{x}_{i,j}\ra_{i,j\in \mu}$ of nice names for elements in $\{0,1\}$ depending on an antichain $A'_{i,j}$ (where $\dot{x}_{i,j}=1$ codes
    $r_i\le_Q r_j$),\footnote{A nice name $\dot{x}$ of a member of $\{0,1\}$ depending on an antichain $C\subseteq\Por$ (allowed to be empty) has the form  $\dot{x}=\{(\check{0},p):p\in C\}$. Note that $p\Vdash\dot{x}=1$ for all $p\in C$, and $q\Vdash\dot{x}=0$ for any $q\in\Por$ incompatible with all the members of $C$. Moreover, $H(\dot{x})=\bigcup_{p\in C}H(p)$.}
    and $B=\bigcup_{i\in\mu} A_i \cup \bigcup_{i,j\in\mu} A'_{i,j}$.
\end{enumerate}
\end{definition}
So in this case
\begin{equation}\label{eq:Hotherspecial}
  H(\dot r)=\bigcup_{p\in A}H(p),\text{ and }
  H(\dot Q)=\bigcup_{p\in B}H(p).
\end{equation}

It is well known that every name of a subset of $\omega$ has
an equivalent nice name. Moreover, as we
can choose the conditions of the antichains in any given dense set, we get the following:
\begin{fact}\label{fact:basicnice}
   (As $\Por$ is ccc)
   Let $D\subseteq \Por$ be dense and let $\lambda$ be a cardinal with uncountable cofinality.
   \begin{enumerate}[(a)]
       \item\label{item:nicereal}
   For any $\Por$-name of a real
   there is an equivalent nice name
   determined by $A\subseteq D$
   with $|A|\le\aleph_0$.
   \item\label{item:niceQ} For any name of a poset of size ${<}\lambda$ consisting of reals,
   there
   is an equivalent nice name
   determined by a set $B\subseteq D$
   with $|B|<\lambda$.
   \end{enumerate}
\end{fact}

Every automorphism of $\Bbf$ induces an automorphism of $\Gor_\Bbf$,
see Lemma~\ref{GBprop1}(\ref{item:notd}).
Therefore, a
$\pi_0$-sequence $h$ of such automorphisms
induces an automorphism of the (FS) product $\Por_{\pi_1}$.
Such an automorphism can sometimes be naturally extended
to the whole iteration $\Por$ (which will allow isomorphism-of-names arguments
and subsequently show $\LCU_\textrm{sp}$).

What do we mean by ``naturally extend''?
Recall that, whenever $f:P\to P$ is an automorphism on some poset $P$, and $\tau$ is a $P$-name, $f$ sends $\tau$ to the $P$-name
\[f^*(\tau):=\{(f^*(\sigma),f(p)):(\sigma,p)\in\tau\}.\]
Also, $(f^{-1})^*(f^*(\tau))=\tau$; and $p\Vdash\varphi(\tau)$ iff $f(p)\Vdash\varphi(f^*(\tau))$
whenever $p\in P$ and $\varphi(x)$ is a formula.
If $\dot{Q}$ is a $P$-name and $P\Vdash f^*(\dot{Q})=\dot{Q}$, then we can certainly
extend $f$ to $P*\dot{Q}$.
We say that $\Por$ is $h$-symmetric, if this is the case
in all steps of the iteration:

\begin{definition}\label{DefAutom}
  Let $\Por$ be a suitable iteration.
    \begin{enumerate}
        \item\label{t2aut} A bijection $h:\pi_1\to\pi_1$ is a \emph{$2$G-automorphism} if, for each $\delta<\pi_0$, $h{\upharpoonright}B_\delta$ is an automorphism of $\Bbf_\delta$.
    \item
    Such an $h$ defines an automorphism
    $\hat{h}_{\pi_1}$ of $\Por_{\pi_1}\to\Por_{\pi_1}$,
    by $\hat{h}_{\pi_1}(p):=\la \hat{f}_\delta(p(\delta)):\delta\in\dom p\ra$
       where $f_\delta:=h{\restriction}[\omega_1\delta,\omega_1(\delta+1))$
       is the automorphism of $\Bbf_\delta$
       induced by $h$, and $\hat f_\delta$ is defined as in
       Lemma~\ref{GBprop1}(\ref{item:notd}).
    \item
    We say
    $\Por$ is $h$-\emph{symmetric} if the following
    inductive construction defines
    $\hat{h}_\xi:\Por_\xi\to\Por_\xi$
    for all $\pi_1\le \xi\le \pi$:
   \begin{enumerate}[(i)]
       \item For $\xi=\zeta+1$, we require that $\Vdash_{\Por_\zeta} {\hat{h}_\zeta}^*(\Qnm_\zeta)=\Qnm_\zeta$.
       (Otherwise the construction fails.)
       We then define $\hat{h}_{\zeta+1}:\Por_{\zeta+1}\to\Por_{\zeta+1}$ by $\hat{h}_{\zeta+1}(p{\upharpoonright}\zeta,p(\zeta))=(\hat{h}_\zeta(p{\upharpoonright}\zeta),{\hat{h}_\zeta}^*(p(\zeta)))$.
       \item For $\xi>\pi_1$ limit, set  $\hat{h}_\xi:=\bigcup_{\zeta<\xi}\hat{h}_\zeta$.
   \end{enumerate}
   In this case set $\hat{h}:=\hat{h}_\pi$, which is an automorphism of $\Por$.


   \item For any
   $\delta<\pi_0$ and any pair $(a,b)\in B_{\delta}$,
   fix a
   $2$G-automorphism
   $h^\delta_{a,b}$ such
   that
   $h^\delta_{a,b}(a)=b$
   and
   $h^\delta_{a,b}{\upharpoonright}\Bbf_\zeta$ is the identity for any $\zeta\ne \delta$.
   We can pick such $h^\delta_{a,b}$
   by Definition~\ref{DefSbG}(vi).

   \item
   Let
   $\Hgroup^*$ be the group generated
   by the
   $h^\delta_{a,b}$
   above. So
   $|\Hgroup^*|=\max\{\pi_0,\aleph_1\}$.
   Note also that for all $h\in \Hgroup^*$ and $\delta\in \pi_0$
   we have $h[B_\delta]=B_\delta$, and that $\supp(h):=\bigcup\{B_\delta:\, h{\restriction}B_\delta\ne\textrm{id}_{B_\delta},\ \delta<\pi_0\}$ has size ${\le}\aleph_1$.
%

  \item
   We say that
   $\Por$ is
   \emph{symmetric} if $\Por$ is $h$-symmetric for every $h\in \Hgroup^*$.
   \end{enumerate}
   %
\end{definition}







In isomorphism-of-names arguments it is relevant to know when a condition or a name remains unchanged after applying an automorphism $\hat h$. The following states a sufficient condition:

\begin{lemma}\label{automident}
    Assume that $\Por$ is $h$-symmetric and $\pi_1\leq\xi\leq\pi$.
\begin{enumerate}[(a)]
    \item
    If $p\in\Por_\xi$ and $h{\upharpoonright}(H(p)\cap \pi_1)$ is the identity, then $\hat{h}_\xi(p)=p$.

    \item
   If $\tau$ is a $\Por_\xi$-name and $h{\upharpoonright}(H(\tau)\cap \pi_1)$ is the identity, then 
   ${\hat{h}_\xi}^*(\tau)=\tau$.

   \item Let $g:=h^{-1}$. Then
   $\Por$ is $g$-symmetric and
   $\hat{g}_\xi=\hat{h}^{-1}_\xi$.
   \end{enumerate}
\end{lemma}
\begin{proof}
    We show the three statements by induction on $\xi$.

    For (a), we use a case
    distinction: Assume
    $\xi=\pi_1$. If $p\in\Por_{\pi_1}$ then $H(p)=\supp p$, and whenever $h$ is the identity on $\supp p$, it is clear that $\hat{h}_{\pi_1}(p)=p$.
    The limit step is also immediate (there are no new conditions, and for names use $\in$-induction).

    For the successor step $\xi=\zeta+1$, assume $p\in\Por_{\zeta+1}$ and that $h$ is the identity on $H(p)\cap\pi_1$. If $\zeta\notin\supp p$, then we have $p\in\Por_\zeta$, so $\hat{h}_{\zeta+1}(p)=\hat{h}_\zeta(p)=p$ by the induction hypothesis. So assume $\zeta\in\supp p$. Then $H(p)=H(p{\upharpoonright}\zeta)\cup\{\zeta\}\cup H(p(\zeta))$, so by induction hypothesis $\hat{h}_\zeta(p{\upharpoonright}\zeta)=p{\upharpoonright}\zeta$ and 
    ${\hat{h}_\zeta}^*(p(\zeta))=p(\zeta)$, thus $\hat{h}_{\zeta+1}(p)=p$.

   We now show (b) by
   $\in$-induction on $\tau$. If $(\sigma,p)\in\tau$ then $H(\sigma)\cup H(p)\subseteq H(\tau)$, so by induction hypothesis and (a), 
   ${\hat{h}_\xi}^*(\sigma)=\sigma$ and $\hat{h}_\xi(p)=p$. Hence
   \[{\hat{h}_\xi}^*(\tau)=\{({\hat{h}_\xi}^*(\sigma),\hat{h}_\xi(p)):(\sigma,p)\in\tau\}
   =\{(\sigma,p):(\sigma,p)\in\tau\}=\tau\]

   For (c), the steps $\xi=\pi_1$ and $\xi>\pi_1$ limit are easy, so we deal with the successor step $\xi=\zeta+1$. So assume that $\hat{g}_\zeta$ is defined and $\hat{g}_\zeta=\hat{h}^{-1}_\zeta$. Since $\hat{h}_\xi$ is defined, $\Vdash_{\Por_\zeta}\hat{h}^*_\zeta(\Qnm_\zeta)=\Qnm_\zeta$, which implies $\Vdash_{\Por_\zeta}\hat{g}_\zeta^*(\Qnm_\zeta)=\Qnm_\zeta$, so $\hat{g}_{\zeta+1}$ is defined; and for any $p\in\Por_\xi$, $\hat{g}_\xi(\hat{h}_\xi(p))=(\hat{g}_\zeta(\hat{h}_\zeta(p{\upharpoonright\zeta})),\hat{g}_\zeta^*(\hat{h}_\zeta^*(p(\zeta))))=p$, so $\hat{g}_\xi=\hat{h}^{-1}_\xi$.
\end{proof}

\subsection{A digression: Self-indexed products}

How to construct a symmetric iteration $\Por$?
We have to make sure that
at each step $\zeta$ the iterand
$\Qnm_\zeta$ is invariant under
$\hat h$
for all $h\in \Hgroup^*$.
One case that will be useful:
$\Qnm_\zeta$ is
a (ccc) FS product such that
whenever
$\dot Q$ is one of the factors,
then $\hat h^*(\dot Q)$
is also one.

But there is a technical difficulty here:
We need $\Vdash_{\Por_\zeta} \hat h^*(\Qnm_\zeta)=\Qnm_\zeta$
(i.e., really equality, not just isomorphism;
as we want to get an actual automorphism
of $\Por_{\zeta+1}$).
This is not possible if we ``naively'' index the product with an ordinal.
For example,
assume $\dot Q_0$, $\dot Q_1$ are such that
$\Vdash_{\Por_\zeta} \hat h^*(\dot Q_i)=\dot Q_{1-i}\ne \dot Q_i$.
Then $\dot Q_0\times \dot Q_1$ (the product with index set $\{0,1\}$)
is not a valid choice for
$\Qor_\zeta$, as
$\Vdash_{\Por_\zeta} \hat h^*(\dot Q_0\times \dot Q_1)=\dot Q_1\times \dot Q_0\ne \dot Q_0\times \dot Q_1$.

So instead, we define (in the extension) the FS
product $\prod \mathcal F$
of a set $\mathcal F$ of posets as
the set of all finite partial
functions $p$ from $\mathcal F$ into $\bigcup \mathcal F$ satisfying $p(Q)\in Q$ for all $Q\in \dom(p)$.
We call this object the \emph{self-indexed product} of the
set $\mathcal F$.




In our framework, we start with
a ground model set $\Xi_\zeta$
of $\Por_\zeta$-names of posets. In the $\Por_\zeta$-extension
we let $\mathcal F$
be the set of evaluations of the names in $\Xi_\zeta$,
and let
$\Qnm_\zeta$ be the self-indexed product of $\mathcal F$.

Assume that all automorphisms from $\Hgroup^*$ can be extended up to $\zeta$.\footnote{For this part all the properties of $\Hgroup^*$ are not required; it is just enough that $h\in\Hgroup^*$ implies $h^{-1}\in\Hgroup^*$.}
We assume  that $\Xi_\zeta$ is closed under
each $h\in \Hgroup^*$, i.e.,  $\dot Q\in \Xi_\zeta$ implies
$\hat h^*_\zeta(\dot Q)\in \Xi_\zeta$.
So as $\Xi_\zeta$ is also closed under
the inverse of $h$, by Lemma~\ref{automident}(c) we even get
$\hat h^*_\zeta[\Xi_\zeta]=\Xi_\zeta$. So in particular
$\hat h^*_\zeta[\Xi_\zeta]$ and $\Xi_\zeta$ evaluate to the
same set and thus yield the same self-indexed product, i.e.,
$\Vdash \hat h^*_\zeta(\Qnm_\zeta)= \Qnm_\zeta$.

We record this fact for later reference:
\begin{fact}\label{fact:specialH1}
Assume that $\Qnm_\zeta$ is a ``self-indexed'' product of $\Xi_\zeta$,
and that
    $\dot Q\in \Xi_\zeta$ implies
    $\hat h^*_\zeta(\dot Q)\in \Xi_\zeta$
    for all $h\in \Hgroup^*$.
    Then $\Por$ forces $\hat h^*_\zeta(\Qnm_\zeta)= \Qnm_\zeta$,
    so we can extend each $h\in\Hgroup^*$ to $\Por_\zeta*\Qnm_\zeta$.
\end{fact}

We additionally assume that each factor is (forced to be)
a real-number-poset.
Assume that $(p,q)\in \Por_\zeta*\Qnm_\zeta$.
We can densely assume that $p$ decides
the finite domain of $q$, more specifically,
the\footnote{Or rather: \emph{a} finite set,
as different  names in $\Xi_\zeta$
might evaluate to the same object, i.e., index.}  finite set $y\subseteq \Xi_\zeta$
such that $p$ forces that $\dom(q)$
is (the set of evaluations of) $y$.
Also, for each $\dot Q\in y$,
we can assume that
$q(\dot Q)$ is a
nice name for a real, determined by some $A_{\dot Q}$.
As usual, we can use a given dense set $D
\subseteq \Por_\zeta$ instead of $\Por_\zeta$.

For later reference:
\begin{fact}\label{fact:specialH2}
Assume that $\Qnm_\zeta$ is a ``self-indexed'' product of $\Xi_\zeta$ as described above, that each
factor is forced to be a set of reals, and that
$D\subseteq \Por_\zeta$ is dense.
    If $(p,q)\in \Por_\zeta*\Qnm_\zeta$,
    then there is a $(p',q')\leq (p,q)$
    such that $p'\in D$ decides the
    (finite) $\dom(q')$, and each $q'(\dot Q)$
    is a nice name determined by some
    $A_{\dot Q}\subseteq D$. So in particular
    \begin{equation}\label{eq:specialH}
H((p',q')) = H(p')\cup \{\zeta\} \cup \bigcup_{\dot Q\in\dom(q')}\bigg(H(\dot Q)\cup \bigcup_{r\in A_{\dot Q}} H(r)\bigg).
    \end{equation}
\end{fact}

\begin{remark}\label{rem:lambdarequired}
   This is the reason hereditarily countable nice names
   are not sufficient in our setting to describe reals:
   Even the \emph{index} $\dot Q$
   in such a  product $\Qnm_\zeta$
   is too complicated.
   However, as all the self-indexed products $\Qnm_\zeta$ we use
   will have factors $\dot Q$ of size ${<}\lambda$, it turns
   out we can restrict ourselves to hereditarily ${<}\lambda$-names (this will be the dense set $\Por^*_\zeta$
   of Definition~\ref{DefSstidy}).
\end{remark}

\subsection{Symmetric small history iterations preserve splitting families}

We are finally ready to prove the central fact about preservation of splitting families.



\begin{definition}\label{Defnice}
Let $\lambda$ be an uncountable cardinal.
\begin{enumerate}[(1)]
    \item A condition
    $q\in \Por$ is \emph{$\lambda$-small}, if
    $|\{\delta<\pi_0:H(q)\cap B_\delta\neq\emptyset\}|<\lambda$.
    \item
   A suitable iteration $\ttbf$ has \emph{$\lambda$-small history} if, for any $p\in\Por$, there is a $\lambda$-small $q\le p$.
\end{enumerate}
\end{definition}

So in particular if $\Por$ has $\lambda$-small history and
$\dot x$ is a name of a subset of $\omega$,
then there is an equivalent nice name $\dot b$
which only  uses $\lambda$-small conditions; and
if $\mu\ge\lambda$ has uncountable cofinality, then
\begin{equation}\label{eq:nice}
|\{\delta<\pi_0:H(\dot b)\cap B_\delta\neq\emptyset\}|<\mu.
\end{equation}

\begin{theorem}\label{PresSplit}
    Let $\Por$ be a
    symmetric
    suitable iteration with $\lambda$-small history.
    Assume $\aleph_1\leq\lambda\leq\mu\le \pi_0$ are cardinals with $\mu$ regular.
    Then $\LCU_{\Rbsp}(\Por_\pi,\mu)$ holds, and it is witnessed by $\{\dot{\eta}_{\omega_1\delta}: \delta<\mu\}$.
\end{theorem}
\begin{proof}
   Towards a contradiction, assume that there are $p\in\Por$ and a $\Por$-name $\dot{b}$ of an infinite subset of $\omega$ such that \[p\Vdash|\{\delta<\mu:\dot{\eta}_{\omega_1\delta}{\upharpoonright}\dot{b}\text{\ is eventually constant}\}|=\mu.\]
   Find $F\in[\mu]^\mu$, $n_0<\omega$ and $e\in\{0,1\}$ such that, for any $\delta\in F$, there is some $p_\delta\leq p$ in $\Por$ such that $\omega_1\delta\in\supp(p_\delta)$ and $p_\delta\Vdash\dot{\eta}_{\omega_1\delta}{\upharpoonright}(\dot{b}\menos n_0)\equiv e$.

   We can assume that
   $\dot b$ is a nice name, more particularly that~\eqref{eq:nice} holds,
   and we can also assume that $p$ is $\lambda$-small.
   So there is some $\delta_0\in F$ such that $B_{\delta_0}\cap(H(p)\cup H(\dot{b}))=\emptyset$.

   Put $a:=\omega_1\delta_0\in B_{\delta_0}$. By~\eqref{eq:S2G}, $a$ is contained in
   an uncountable $R_{\delta_0,e}$-complete $U\subseteq B_{\delta_0}$.
   Recall that by the definition of ``symmetric'',
   there is for each $c\in U$ a $2$G-automorphism $h^c\in \Hgroup^*$ such that
   $h^c(a)=c$ and
   such that $h^c{\upharpoonright}\Bbf_\delta$ is the identity for all $\delta\neq\delta_0$. 
   Hence, by Lemma~\ref{automident}, $\hat{h}_\pi^c(p)=p$ and $(\hat{h}_\pi^c)^*(\dot{b})=\dot{b}$, therefore $p'_c:=\hat{h}_\pi^c(p_{\delta_0})\leq p$ and, since $(\hat{h}_\pi^c)^*(\dot{\eta}_a)=\dot{\eta}_c$,
   \[p'_c\Vdash\dot{\eta}_c{\upharpoonright}(\dot{b}\menos n_0)\equiv e.\]
   Lemma~\ref{GBprop1}(\ref{item:dafterall}) implies that $\la p'_c:c\in U\ra$ must be an antichain, which contradicts that $\Por_\pi$ is ccc.
\end{proof}

\begin{remark}
   The same argument shows that, for any $g\in\prod_{\delta<\mu}B_\delta$ (in the ground model), $\{\dot{\eta}_{g(\delta)}:\delta<\mu\}$ witnesses $\LCU_{\Rbsp}(\Por_\pi,\mu)$.
\end{remark}


\section{Suslin-\texorpdfstring{$\lambda$}{lambda}-small iterations}\label{sec:main1}

\newcommand{\ac}{\mathrm{ac}}
\newcommand{\anm}{\mathrm{an}}
\newcommand{\cola}{\mathrm{col}}
\newcommand{\up}{\mathrm{up}}
\newcommand{\op}{\mathrm{op}}
\newcommand{\unnm}{\mathrm{un}}
\newcommand{\nice}{\mathrm{nice}}
\newcommand{\ncp}{\mathrm{ncp}}

We now investigate suitable iterations where the iterand
$\Qnm_\zeta$
at step $\zeta>0$ (i.e., after the initial FS product) is
\begin{enumerate}
    \item
either a \emph{restricted (also called: partial)} Suslin ccc poset (e.g., random forcing
evaluated in some $V^{\Por^-_\zeta}$ for some complete
subforcing $\Por^-_\zeta$ of $\Por_\zeta$);
\item
or the FS product of (in our application: at most $|\pi|$-many)
${<}\lambda$-size posets of reals.
\end{enumerate}

More formally:

\begin{definition}\label{DefSs}
   Let $\lambda$ be an uncountable cardinal. A
   \emph{Suslin-$\lambda$-small} iteration (abbreviated S$\lambda$s) is a suitable iteration $\Por$ with the following properties:
   \begin{enumerate}[({S}1)]
       \item $\pi\menos\pi_1$ is partitioned into two sets $\Sigma^\tbf$ and $\Pi^\tbf$. 




       \item\label{item:souslin} For $\xi\in \Sigma^\tbf$,
          \begin{enumerate}[(i)]
          \item $\Por^{-}_\xi$ is a complete subposet of $\Por_\xi$,
              \item $\Sor^\tbf_\xi$ is a definition of a Suslin ccc poset (with parameters in the ground model),
              \item $\Qnm_\xi$ is a $\Por_\xi$-name for $(\Sor^\tbf_\xi)^{V^{\Por^{-}_\xi}}$.
          \end{enumerate}

       \item\label{item:prod} 
       %
       For $\xi\in \Pi^\tbf$,
       \begin{enumerate}[(i)]
        \item $\Xi^\tbf_\xi$ is a set
        in the ground model,
        \item each element of $\Xi^\tbf_\xi$
        is a
        $\Por_\xi$-name
         $\dot Q $
       for a poset of size\footnote{That is, each element of $\Xi^\tbf_\xi$ is forced to have size ${<}\lambda$, whereas the cardinality of $\Xi^\tbf_\xi$ may be as large as we want.}
       ${<}\lambda$ consisting of reals,
       \item
       $\Qnm_\xi$ is (the $\Por_\xi$-name for) the FS product of $\Xi_\xi$.
       \end{enumerate}


   \end{enumerate}
\end{definition}



\begin{remark}
   Regarding (S\ref{item:prod}), recall that
   our setting  requires
   $\Qnm_\xi$ (to be forced by $\Por_\xi$) to be ccc (as suitable iterations have to be ccc).
   In contrast, in (S\ref{item:souslin}),
   $\Qnm_\xi$ will be always ccc ``for free''
   (in $V^{\Por^{\tbf}_\xi}$ as well as in $V^{\Por^{-}_\xi}$), as it is an evaluation of a Suslin ccc definition (see~\cite{JSsuslin}).
\end{remark}

We now show that we can replace such an
iteration $\la \Por'_\zeta,\Qnm'_\zeta:\zeta\in\pi\ra$ with
an \emph{isomorphic} version
$\la\Por_\zeta,\Qnm_\zeta:\zeta\in\pi\ra$:
The only difference will
be in steps $\zeta\in \Pi$, where we
select (hereditarily) nice names for the factors
$\dot Q\in \Xi_\zeta$ and make sure that $\Qnm_\zeta$ is self-indexed.
In addition, we will define
a dense subset $\Por^*$ of hereditarily $\lambda$-small
conditions, an extended
``refined history domain'' $\pi^+$, and
a ``refined history'' $H^*:\Por^*\to \mathcal P(\pi^+)$. These are formalized in the following notions.

\begin{definition}\label{DefSstidy}
   Let $\lambda$ be an uncountable cardinal. A
   \emph{tidy Suslin-$\lambda$-small} iteration is a Suslin-$\lambda$-small iteration $\Por$ with the following additional components and properties:
   \begin{enumerate}[(1)]
      \item For $\xi\in\pi\menos\pi_1$, $\Por^*_\xi$ is a dense subset of $\Por^\tbf_\xi$.

      \item $\Por^*_{\pi_1}=\Por_{\pi_1}$.


       \item If $\xi\in\Sigma$ and $p\in\Por^*_{\xi+1}$ then $p{\restriction}\xi\in\Por^*_\xi$, $p(\xi)$ is a nice $\Por^*_\xi$-name of a real and $\Vdash_{\Por_\xi}p(\xi)\in\Sor_\xi\cap V^{\Por^-_\xi}$.

       \item For $\xi\in \Pi^\tbf$, $\Xi_\xi$ is composed of nice $\Por^*_\xi$-names for real-number-posets of size ${<}\lambda$ (See Definition~\ref{Defnicenm}(2)).

       In addition, if $p\in\Por^*_{\xi+1}$ then the following is satisfied:
       \begin{enumerate}[(i)]
           \item $p{\upharpoonright}\xi\in\Por^*_\xi$.
           \item $\dom p(\xi)$ is decided by $p{\upharpoonright}\xi$, that is, $p{\upharpoonright}\xi\Vdash_{\Por^*_\xi}$``$\dom p(\xi)=d^p_\xi$" for some finite $d^p_\xi\subseteq\Xi^\tbf_\xi$.
           \item For each $\dot{Q}\in\dom p(\xi)$, $p(\xi,\dot{Q})$ is a nice $\Por^*_\xi$-name of a real and $\Vdash_{\Por^*_\xi}p(\xi,\dot{Q})\in\dot{Q}$.
           \item $p(\xi)=\la p(\xi,\dot{Q})):\dot{Q}\in\dom p(\xi)\ra$ (in particular, $p(\xi)$ is a $\Por^*_\xi$-name).
       \end{enumerate}

       \item If $\pi_1\leq\xi<\pi$ then $\Por^*_\xi\subseteq\Por^*_{\xi+1}$.

       \item If $\gamma\in(\pi_1,\pi]$ is limit then $\Por^*_\gamma=\bigcup_{\xi<\gamma}\Por^*_\xi$.
   \end{enumerate}
   Denote $\Por^*:=\Por^*_\pi$.
\end{definition}

Note that tidy S$\lambda$s iterations are coherent in the sense that $\Por^*_\eta\cap\Por_\xi=\Por^*_\xi$ for any $\pi_1\leq\xi\leq\eta\leq\pi$. Conditions (5) and (6) were included to guarantee this.

\begin{definition}\label{def:refhistory}
   Let $\Por$ be a tidy S$\lambda$s iteration.
   \begin{enumerate}[(1)]
       \item For $\pi_1\leq\xi\leq\pi$ define the \emph{refined history domain} $\xi^+=\xi\cup\bigcup_{\zeta\in\xi\cap\Pi}\{\zeta\}\times\Xi_\zeta$.

       \item For $p\in\Por^*$ and a $\Por^*$-name $\tau$ we define the \emph{refined history} $H^*(p)\subseteq\pi^+$ and $H^*(\tau)\subseteq\pi^+$ as follows. For $\pi_1\leq\xi\leq\pi$ we define $H^*$ by recursion on $\xi$ for $p\in\Por^*_\xi$ and for a $\Por^*_\xi$-name $\tau$.
       \begin{enumerate}[(i)]
           \item For $p\in\Por^*_{\pi_1}$, $H^*(p):=H(p)$.
           \item For $\xi=\zeta+1$ and $p\in\Por^*_{\zeta+1}$,
                $H^*(p)=H^*(p{\restriction}\zeta)$ when $\zeta\notin\supp p$, otherwise:
                \begin{itemize}
                    \item if $\zeta\in\Sigma$ then
                    \[H^*(p):=H^*(p{\restriction}\zeta)\cup\{\zeta\}\cup H^*(p(\zeta));\]

                    \item if $\zeta\in\Pi$ then
                    \[H^*(p):=H^*(p{\restriction}\zeta)\cup\{\zeta\}\cup(\{\zeta\}\times\dom(p(\zeta)))\cup\bigcup_{\dot{Q}\in\dom(p(\zeta)}(H^*(\dot{Q})\cup H^*(p(\zeta,\dot Q)).\]
                \end{itemize}

            \item When $\xi>\pi_1$ is limit and $p\in\Por^*_\xi$, then $H^*(p)$ has already been defined (because $p\in\Por_\zeta$ for some $\zeta<\xi$).

            \item For any $\Por^*_\xi$-name $\tau$ define, by $\in$-recursion,
                 \[H^*(\tau):=\bigcup\{H^*(\sigma)\cup H^*(p):(\sigma,p)\in\tau\}.\]
       \end{enumerate}
   \end{enumerate}
\end{definition}

Tidy S$\lambda$s iterations have many features that ease its manipulation, in particular, they have $\lambda$-small history.

\begin{lemma}\label{smallH}
    Let $\Por$ be a tidy S$\lambda$s iteration with $\lambda$ regular. Then, for any $p\in\Por^*$:
    \begin{enumerate}[(a)]
        \item $|H^*(p)|<\lambda$.
        \item $H(p)=H^*(p)\cap\pi$.
        \item $H(\tau)=H^*(\tau)\cap\pi$ for any $\Por^*_\pi$-name $\tau$.
    \end{enumerate}
    In particular, $\Por$ has $\lambda$-small history.
\end{lemma}
\begin{proof}
   We prove (a), (b) and (c) simultaneously for all $p\in\Por^*_\xi$ by recursion on $\pi_1\leq\xi\leq\pi$. It is clear that (c) follows from (b).

   In the case $\xi=\pi_1$, $H^*(p)=\supp p=H(p)$, which is finite.

   For the successor step $\xi=\zeta+1$, assume $\zeta\in\supp p$. If $\zeta\in\Sigma$ then $p(\zeta)$ is a nice $\Por^*_\zeta$-name of a real, so it is determined by some countable $A\subseteq \Por^*_\zeta$. Hence
  \[
     H^*(p) = H^*(p{\restriction}\zeta)\cup\{\zeta\}\cup
     \bigcup\{H^*(r):\, r\in A\}
   \]
   so, by induction hypothesis, $|H^*(p)|<\lambda$.

   Now assume $\zeta\in \Pi$. Since any $\dot Q\in\Xi_\zeta$ is a nice $\Por^*_\zeta$-name for a real-number-poset of size ${<\lambda}$, it is determined by some $B_{\dot Q}$ of size ${<\lambda}$. Hence
   \[
       H^*(\dot Q)=\bigcup_{s\in B_{\dot Q}}H^*(s),\text{\ and }|H^*(\dot Q)|<\lambda,
   \]
   the latter by induction hypothesis.
   On the other hand, for any $\dot Q\in\dom p(\zeta)$, $p(\zeta,\dot Q)$ is a $\Por^*_\zeta$-name of a real, so it is determined by some countable $A_{\dot Q}\subseteq\Por^*_\zeta$. Hence
   \[H^*(p(\zeta,\dot Q))=\bigcup_{r\in A_{\dot Q}}H^*(r),\]
   which have size ${<}\lambda$ by induction hypothesis. As
   \[H^*(p)=H^*(p{\upharpoonright}\zeta)\cup\{\zeta\}\cup(\{\zeta\}\times\dom p(\zeta))
                \cup\bigcup_{\dot{Q}\in\dom p(\zeta)}(H^*(\dot{Q})\cup H^*(p(\zeta,\dot{Q}))),\]
   we get $|H^*(p)|<\lambda$. On the other hand, since $p(\zeta)=\la p(\zeta,\dot Q): \dot Q\in\dom p(\zeta)\ra$,
   \[H(p(\zeta))=\bigcup_{\dot{Q}\in\dom p(\zeta)}(H(\dot{Q})\cup H(p(\zeta,\dot{Q}))),\]
   so we can deduce (b). The limit step is immediate.
\end{proof}

As promised, we show that any Suslin-$\lambda$-small iteration is isomorphic to a tidy one.

\begin{lemma}\label{lem:eqvtidy}
   If $\lambda$ is regular uncountable, then any Suslin-$\lambda$-small iteration is isomorphic to a tidy S$\lambda$s iteration.
\end{lemma}

\begin{proof}
   By recursion
   on $\pi_1\leq\xi\leq\pi$
   we construct the tidy iteration up to $\Por_\xi$, along with its components,
   and the isomorphism $i_\xi:\Por'_\xi\to\Por_\xi$. We also guarantee that
   $i_\xi$ extends $i_\zeta$ for any $\pi_1\leq\zeta<\xi$.

   \textbf{Case $\bm{\xi=\pi_1}$:}
   Set $\Por^*_{\pi_1}=\Por_{\pi_1}=\Por'_{\pi_1}$ and let $i_{\pi_1}$ be the identity function.

   \textbf{Case $\bm{\xi=\zeta+1}$ with $\bm{\zeta\in \Sigma}$:}
   As $i_\zeta:\Por'_\zeta\to\Por_\zeta$ is an isomorphism, we let
   $\Por^-_\zeta=i_{\zeta}[\Por^{\prime-}_\zeta]$, which is clearly a complete
   subforcing of $\Por_\zeta$, and evaluate
   $\Sor_\zeta$ accordingly. Note that $i_{\zeta}$ can be extended to
   an isomorphism $i_{\zeta+1}:\Por'_{\zeta+1}\to\Por_{\zeta+1}$ in a natural way.

   We define
   $\Por_{\zeta+1}^*$ as the set of pairs
   $(p,\dot q)$ where $p\in \Por_\zeta^*$,
   $\dot q$ is a $\Por^*_\zeta$ nice name
   for a real, and $p \Vdash \dot q\in \Sor_\zeta\cap V^{\Por_\zeta^-}$.
   This is dense according to Fact~\ref{fact:basicnice}(\ref{item:nicereal}).

   \textbf{Case $\bm{\xi=\zeta+1}$ with $\zeta\in \Pi$:}
   Fix a $\dot Q'$ in $\Xi'_\zeta$.
   As $i_\zeta^*(\dot Q')$ is forced by $\Por_\zeta$ to have size
   ${<}\lambda$,
   according to Fact~\ref{fact:basicnice}(\ref{item:niceQ}),
   %
   there is an equivalent
   $\Por^*_\zeta$-nice name $\dot Q$
   determined by
   $B_{\dot Q}\subseteq \Por^*_\zeta$
   of size ${<}\lambda$.
   Let $\Xi_\zeta$ be the set of all these
   names, and define
   $\Qnm_\zeta$ to be the self-indexed FS product
   of the $\dot Q$ in $\Xi_\zeta$. We can obtain the isomorphism $i_{\xi+1}$ in
   a natural way.

   We define $\Por_{\zeta+1}^*$ to consist
   of the $(p,\dot q)$ as in Fact~\ref{fact:specialH2}
   (using $D=\Por_\zeta^*$).

   \textbf{Case $\xi>\pi_1$ limit:}
   As $\Por$ is a FS iteration, we (have to) set
   $\Por_\xi=\bigcup_{\zeta<\xi}\Por_\zeta$;
   and we set
   $\Por^*_\xi:=\bigcup_{\zeta<\xi}\Por^*_\zeta$ and $i_\xi:=\bigcup_{\zeta<\xi}i_\zeta$.
\end{proof}

For later reference, note:
   Assume that we can extend some $2$G-automorphism $f$ to $\Por_\zeta$, and that
   $\dot Q\in \Xi_\zeta$.
   Set $\supp(f):=\bigcup\{B_\delta:\, f{\restriction} B_\delta\ne\textrm{id}_{B_\delta}\}$.
   Then
   \begin{equation}\label{eq:unclear}
    \supp(f)\cap H^*(\dot Q)=\emptyset\text{ implies }
    \hat f^*_\xi(\dot Q)=\dot Q.
   \end{equation}
   This follows
   from Lemmas~\ref{automident}(b) and~\ref{smallH}(c).

In our applications, $\Por^-_\xi$ has the following form.

\begin{definition}\label{def:supprestr}
   Let $\Por$ be a tidy S$\lambda$s iteration. For any $X\subseteq\pi^+$, define $\Por^*{\upharpoonright}X:=\{p\in\Por^*_\pi: H^*(p)\subseteq X\}$.
\end{definition}

Note that generally $\Por^*{\restriction} X$ will not be a complete subforcing
of $\Por^*$; but we will only be interested in the case where it is, see Lemma~\ref{smallrestr}.

\begin{lemma}\label{item:restr}
    Let $\Por$ be a tidy S$\lambda$s iteration with $\lambda$ uncountable regular, and let $\mu\geq\lambda$ be regular and $\aleph_1$-inaccessible\footnote{Recall that a cardinal $\mu$ is \emph{$\kappa$-inaccessible} if $\theta^\nu<\mu$ for any cardinals $\theta<\mu$ and $\nu<\kappa$.}. If $X\subseteq\pi^+$ and $|X|<\mu$ then
    $|\Por^*{\upharpoonright}X|< \mu$.
\end{lemma}
\begin{proof}
   By induction on $\xi\in[\pi_1,\pi]$ we show that, whenever $X\subseteq\xi^+$ has size ${<}\mu$, $|\Por^*{\upharpoonright}X|<\mu$.
   For $\xi=\pi_1$, it is clear that $|\Por^*{\upharpoonright}X|=\max\{|\omega\times X|,1\}<\mu$.

   For limit $\xi>\pi_1$, $\Por^*{\upharpoonright}X=\bigcup_{\eta\in c}\Por^*{\upharpoonright}(X\cap \eta^+)$ where $c$ is a cofinal subset of $\xi$ of size $\cf(\xi)$.
   If $\cf(\xi)<\mu$ then
   $|\Por^*{\upharpoonright}X|<\mu$ because it is a union of ${<}\mu$ many sets of size ${<}\mu$; if $\cf(\xi)\geq\mu$ then $X\subseteq\eta^+$ for some $\eta<\xi$, so
  $|\Por^*{\upharpoonright}X|<\mu$ by induction hypothesis.

   For the successor step $\xi=\zeta+1$, assume $X\subseteq(\zeta+1)^+$ and $X\nsubseteq\zeta^+$ (the non-trivial case). Put $X_0:=X\cap\zeta^+$. By induction hypothesis,
   $|\Por^*{\upharpoonright}X_0|<\mu$
   and, since $\mu$ is $\aleph_1$-inaccessible,
   there are at most $|\Por^*{\upharpoonright}X_0|^{\aleph_0}<\mu$ many nice $\Por^*{\upharpoonright}X_0$-names of reals.

   Let $p\in\Por^*{\upharpoonright}X$.
   If $\xi\in \Sigma$ then $p(\xi)$ is a nice $\Por^*{\upharpoonright}X_0$-name of a real; if $\xi\in \Pi$, then $p(\xi)$ is determined by a finite partial function from $(\{\xi\}\times\Xi_\xi)\cap X$ into the set of nice $\Por^*{\upharpoonright}X_0$-names of reals, and there are ${<}\mu$-many such finite partial functions. Hence, $|\Por^*{\upharpoonright}X|<\mu$.
\end{proof}

\begin{corollary}\label{cor:small}
Let $\Por$ be a tidy S$\lambda$s iteration
with $\lambda$ regular.
\begin{enumerate}[(a)]
    \item\label{item:quaxa} $\Por$ has $\lambda$-small history.
    \item\label{item:quaxb} Every $p\in \Por^*$
        is an element
        of $\Por^*{\restriction} X$
        from some $X\subseteq\pi^+$
        of size ${<}\lambda$.
        \item\label{item:quaxc}
        For every $\Por$-name
        of a real
        there is an equivalent
        $\Por^*{\restriction} X$-name
        for some $X\subseteq\pi^+$ of size ${<}\lambda$.
    \item\label{item:quaxd} Assume that $|\pi|^{\aleph_0}=|\pi|$, that $\lambda\le |\pi|^+$, and that
        $|\Xi_\zeta|\le |\pi|$ for each $\zeta\in \Pi$.
        Then $|\Por^*|\le |\pi|$.
\end{enumerate}
\end{corollary}
\begin{proof}
   (\ref{item:quaxa})  follows from
   Lemma~\ref{smallH}(a),(b);
   (\ref{item:quaxb}) follows  from
   Lemma~\ref{smallH}(a)
   (using $X:=H^*(p)$); and (\ref{item:quaxc}) follows from (\ref{item:quaxb})
   (use a nice $\Por^*$-name for a real).

   For (\ref{item:quaxd}), set $\mu=|\pi|^+$ (the cardinal successor). Note that $\mu$ is ${<}\aleph_1$-inaccessible because $|\pi|^{\aleph_0}=|\pi|$, and
   $|\pi^+|\le |\pi|\times \sup_{\zeta\in\Pi}\{|\Xi_\zeta|\}\le |\pi|<\mu$.
   Therefore
   $|\Por^*|=|\Por^*{\restriction} \pi^+|<\mu$
   by Lemma~\ref{item:restr}
   (for $X=\pi^+$).
\end{proof}

In our applications, $\Por^-_\xi=\Por^*{\upharpoonright}C_\xi$ with $C_\xi\subseteq\xi^+$
for all $\xi\in\Sigma$.
We will now show how to build symmetric S$\lambda$s-iterations:

\begin{definition}\label{def:closed}
   Let $\Por$ be a tidy S$\lambda$s iteration, and let $h:\pi_1\to\pi_1$ be a $2$G-automorphism.
   \begin{enumerate}[(1)]
       \item Let $\xi\in\Pi$. We say that \emph{$\Xi_\xi$ is closed}, if, whenever $h\in\Hgroup^*$ and $\hat{h}_\xi$ can be defined (see Definition~\ref{DefAutom}), $\dot Q\in\Xi_\xi$ implies $\hat{h}^*_\xi(\dot Q)\in\Xi_\xi$
       (where $\Hgroup^*$ is the group of $2$G-automorphisms fixed in Definition~\ref{DefAutom}(5)).

       \item  We say that $C\subseteq\pi^+$ is \emph{closed} if, for any $h\in\Hgroup^*$, it satisfies:
   \begin{enumerate}[(i)]
       \item For any $\delta<\pi_0$, $B_\delta\cap C\neq\emptyset$
        implies $B_\delta\subseteq C$.
       \item For any $\xi\in\Pi$, whenever $\hat{h}_\xi$ can be defined, if $(\xi,\dot{Q})\in C$ then $(\xi,\hat{h}^*_\xi(\dot Q))\in C$ and $\xi\in C$.
   \end{enumerate}
   \end{enumerate}
\end{definition}

\begin{lemma}\label{lem:kjgeqwtr}
   Assume that $\Por$ is a tidy S$\lambda$s iteration
   such that
   the following requirements are satisfied:
   \begin{enumerate}[(I)]
       \item\label{item:blaprod} For any $\xi\in\Pi$, $\Xi_\xi$ is closed.
       \item\label{item:blasigma} For any $\xi\in\Sigma$, $\Por^-_\xi=\Por^*{\restriction}C_\xi$ where $C_\xi\subseteq\xi^+$ is closed.
   \end{enumerate}
   Then we get:
   \begin{enumerate}[(a)]
       \item\label{item:symmetric}
           $\Por$ is symmetric (i.e., $h$-symmetric for all $h\in\Hgroup^*$).
     \item $\hat h[\Por^*{\restriction} C]=
           \Por^*{\restriction} C$ for all
           closed $C\subseteq \pi^+$.
   \end{enumerate}
\end{lemma}

Of course we will have to also
make
sure that the assumption ``$\Por$ is
a tidy S$\lambda$s-iteration'' is satisfied.
The nontrivial points of these assumptions are:
\begin{enumerate}[(I{-b})]
\item
For $\zeta\in \Pi$ the
FS product
$\Qor_\zeta$ is ccc.
\\
(In our case this will be trivial,
as all factors $\dot Q$ will be Knaster).
\item
For $\zeta\in \Sigma$,
$\Por^*{\restriction} C_\zeta$
is a complete subforcing of
$\Por^*$.
\\
We will see in Lemma~\ref{smallrestr}
how to achieve this.
\end{enumerate}

\begin{proof}
   By induction on $\xi\in[\pi_1,\pi]$ we show that $\hat{h}_\xi$ can be defined for any $h\in\Hgroup^*$ (towards (a)), and that (b) is valid for any closed $C\subseteq\xi^+$.\footnote{In this proof we only use that $h\in\Hgroup^*$ implies $h^{-1}\in\Hgroup^*$.}

   \textbf{Case $\bm{\xi=\pi_1}$:} It is clear that $\hat{h}_{\pi_1}$ can be defined;
   (b) is clear because $h[B_\delta]=B_\delta$ for any $\delta<\pi_0$ and $h\in\Hgroup^*$.

   \textbf{Case $\bm{\xi=\zeta+1}$ with $\bm{\zeta\in \Sigma}$:}
   (a): Note that
   $\Vdash\hat h^*_\zeta(\Qor_\zeta)=
   \hat h^*_\zeta(\Sor_\zeta^{V^{\Por^*{\restriction} C_\zeta}})=\Sor_\zeta^{V^{\hat h_\zeta[\Por^*{\restriction} C_\zeta]}}$
   (as $\Sor_\zeta$ only uses parameters from the ground model), which is $\Qnm_\zeta$
   by induction hypothesis (as we assume that $C_\zeta$
   is closed).
   So we can extend $\hat h_\zeta$ to $\hat h_{\zeta+1}$.

   (b): Note that $\xi^+=\zeta^+\cup\{\zeta\}$,
   so if $C\subseteq \xi^+$ is closed
   (and not already a subset of $\zeta^+$),
   then
   $C=C'\cup \{\zeta\}$ with $C'$ closed,
   and $(p,\dot q) \in\Por^*{\restriction} C$
   means that $p\in\Por^*{\restriction}C'$ and
   $\dot q$ is a nice
   $\Por^*{\restriction} C'$-name
   for an element of $\Qnm_\zeta$.
   Then $\hat h^*_\zeta(\dot q)$
   is a nice $\hat h_\zeta[\Por^*{\restriction} C']$-name
   for an element of
   $\hat h^*_\zeta(\Qnm_\zeta)$,
   which is by induction hypothesis a nice
   $\Por^*{\restriction} C'$-name
   for an element of
   $\Qnm_\zeta$, i.e., $\hat h_{\zeta+1}((p,\dot q))\in\Por^*{\restriction} C$.
   This shows that $\hat h_\xi[\Por^*{\restriction} C]\subseteq \Por^*{\restriction} C$.
   As this is also true for the inverse
   of $h$ (because $h^{-1}\in\Hgroup^*$), by Lemma~\ref{automident}(c) we get equality.

   \textbf{Case $\bm{\xi=\zeta+1}$ with $\bm{\zeta\in \Pi}$:}  First note that, by induction hypothesis, $\hat h_\zeta[\Por^*_\zeta]=\Por^*_\zeta$ because $\Por^*_\zeta=\Por^*{\restriction}\zeta^+$ and $\zeta^+$ is closed.

   (a): Since
   $\Qnm_\zeta$
   is a $\Por^*_\zeta$-name
   for the self-indexed FS product
   of the (evaluated)
   set $\Xi_\zeta=\{\dot Q:\, \dot Q\in \Xi_\zeta\}$,
   $\hat h^*_\zeta(\Qnm_\zeta)$ is the
   $\Por^*_\zeta$-name
   for the self-indexed FS product
   of the (evaluated) set
   $\hat h^*_\zeta[\Xi_\zeta]=\{\hat h^*_\zeta(\dot Q):\, \dot Q\in \Xi_\zeta\}$.
   But as the ground model set of names
   $\{\hat h^*(\dot Q):\, \dot Q\in \Xi_\zeta\}$ is identical to the
   ground model set of names $\Xi_\zeta$,
   their evaluations are identical as well (because $\Xi_\xi$ is closed under inverses, and by Lemma~\ref{automident}(c)).
   In other words, $\hat h^*_\zeta(\Qnm_\zeta)=\Qnm_\zeta$, and $\hat h_\zeta$ can be extended to $\Por_\xi$.

   (b): Assume that $C\subseteq \xi^+=\zeta^+\cup\{\zeta\}\cup(\{\zeta\}\times \Xi_\zeta)$
   is closed, and that
   $(p,\dot q)\in \Por^*{\restriction} C$ (but to avoid the trivial case, not in $\Por^*_\zeta$).
   This means that
    $p\in \Por^*{\restriction}
   C'$ for $C'=C\cap \zeta^+$,
   and it determines $\dom(\dot q)=\{\dot Q_1,\dots , \dot Q_n\}$, such that all $(\zeta, \dot Q_i)$ are in $C$ (and each
   $\dot q(\dot Q_i)$ is a nice $\Por^*{\restriction}C'$-name).
   Then $\hat h_\zeta(p)\in \Por^*{\restriction} C'$
   by induction hypothesis, and it determines
   $\dom(\hat h^*_\zeta(\dot q))=\{\hat h^*_\zeta(\dot Q_1),\dots , \hat h^*_\zeta(\dot Q_n)\}$ (and each
   $\hat h^*_\zeta(\dot q)(\dot Q_i)$ is a $\Por^*{\restriction}C'$-nice name).
   Accordingly $\hat h_\xi((p,\dot q))\in \Por^*{\restriction} C$ as required.

   We conclude that
   $\hat h_\xi[\Por^*{\restriction}C]\subseteq\Por^*{\restriction}C$, but equality holds because the same is true for $h^{-1}$.

   \textbf{Case $\bm{\xi}$ limit:}
   By induction hypothesis, $\hat{h}_\zeta$ is defined for all $\zeta<\xi$, so $\hat h_\xi$ is defined (as its union). On the other hand, if $C\subseteq\xi^+$ is closed then $C=\bigcup_{\zeta<\xi}C\cap\zeta^+$ where each $C\cap\zeta^+$ is closed, so $\hat h_\xi[\Por^*{\restriction}C]=\Por^*{\restriction}C$ by induction hypothesis.
\end{proof}

We address some few facts about closed sets.

\begin{lemma}\label{lem:closure}
    Let $\Por$ be a symmetric tidy S$\lambda$s iteration. Then:
    \begin{enumerate}[(a)]
        \item\label{item:sjlgclosed} The union of closed sets
     is closed.

        \item
     If $A\subseteq \pi^+$ has size ${<}\mu$, with $\mu\ge\max\{\lambda,\aleph_2\}$ uncountable regular,
     then the closure $\overline A$ of $A$
     (the smallest closed set containing $A$)
     has size ${<}\mu$.
    \end{enumerate}
\end{lemma}
\begin{proof}
   Property (a) is straightforward. We show by induction on $\xi\in[\pi_1,\pi]$ that (b) holds for any $A\subseteq\xi^+$.

   \textbf{Case $\bm{\xi=\pi_1}$:}
   $\overline A=\bigcup\{B_\delta:\, B_\delta\cap A\neq \emptyset\}$. So
   $|\overline A|=\aleph_1\times |A|<\mu$.

   \textbf{Case $\bm{\xi=\zeta+1}$ with $\bm{\zeta\in \Sigma}$:}
   If $A\subseteq \xi^+$ has size
   ${<}\mu$,
   then $\overline A\subseteq
   \overline{(A\cap \zeta^+)}\cup \{\zeta\}$ has size ${<}\mu$.

   \textbf{Case $\bm{\xi=\zeta+1}$ with $\bm{\zeta\in \Pi}$:}
   For $h\in\Hgroup^*$, set  $\supp(h)=\bigcup\{B_\delta:\, h{\restriction} B_\delta\neq \textrm{id}_{B_\delta}\}$.
   Let $A^*$ be the closure of
   \[
     \bigcup\{ H^*(\dot Q)\cap\pi_1:\, (\zeta,\dot Q)\in A\}.
   \]
   $H^*(\dot Q)$ has size ${<}\lambda \le \mu$ by Lemma~\ref{smallH}(c),
   and therefore also the set $A^*$ has size  ${<}\mu$.

   It is clear that
   $\overline A\subseteq
   \overline{(A\cap \zeta^+)}\cup \{\zeta\}
   \cup X$ for
   \[
     X:=\{(\zeta, \hat h^*(\dot Q)):\, (\zeta,\dot Q)\in A, h\in\Hgroup^*\}.
   \]
   Since $\Hgroup^*$ is a group, $\{\zeta\}\cup X$ is closed. We claim that we get the same set $X$ if we replace
   $\Hgroup^*$ with
   \[
     \Hgroup':=\{g\in \Hgroup^*:\, \supp(g)\subseteq A^*\}.
   \]
   As $\Hgroup'$ has size ${<}\mu$ (recall that $|\supp(g)|\leq\aleph_1$ for any $g\in\Hgroup^*$), we get
   $|\overline A|<\mu$, as required.

   Note that for $f\in \Hgroup^*$ and $(\zeta,\dot Q)\in A$,
   by~\eqref{eq:unclear}
   $\supp(f)\cap A^*=\emptyset$ implies $\hat f^*_\zeta(\dot Q)=\dot Q$.
   And for $h\in \Hgroup^*$, there is a
   $g\in \Hgroup'$ such that $f:= g^{-1}\circ h$
   satisfies $\supp(f)\cap A^*=\emptyset$. (Basically, $g{\restriction}A^*=h{\restriction} A^*$ and $g{\restriction}(\pi_1\menos A^*)$ is the identity.)
   So $\hat f^*(\dot Q)=\dot Q$, which implies
   \[
     \hat g^*(\dot Q)=\hat g^*(\hat f^*(\dot Q))=\hat h^*(\dot Q),
   \]
   as required.

   \textbf{Case $\bm{\xi}$ limit:}
   If $\cof(\xi)\ge \mu$, then
   $A\subseteq \zeta^+$ for some $\zeta<\xi$,
   so $|\overline A|<\mu$ by induction hypothesis.
   Otherwise, $\overline A=\bigcup_{\zeta\in I}\overline{A\cap \zeta^+}$ for some witness $I$ of $\cof(\xi)$, so again
   $|\overline A|<\mu$.
\end{proof}

\begin{lemma}\label{lem:allthefactors}
To satisfy assumption (\ref{item:blaprod}) of Lemma~\ref{lem:kjgeqwtr} for $\xi\in \Pi$,
the following is sufficient, while assuming (I) and (II) for $\zeta<\xi$:
\begin{enumerate}[(a)]
    \item[(\ref{item:blaprod}')]
For some formula $\varphi(x,y)$
using only parameters from the
ground model and some $\kappa_\xi\leq\lambda$,
$\Xi_\xi$ is the set of
all nice $\Por^*_\xi$-names
$\dot Q$ for ${<}\kappa_\xi$-sized
forcings consisting of reals such that
$\Vdash_{\Por^*_\xi} \varphi(\dot Q,\xi)$.
\end{enumerate}
\end{lemma}

\begin{proof}
   By the assumption and Lemma~\ref{lem:kjgeqwtr}, $\Por_\xi$ is symmetric and
   $h^*_\xi[\Por_\xi^*]=\Por_\xi^*$ for any $h\in\Hgroup^*$ (because $\Por_\xi^*=\Por^*{\upharpoonright}\xi^+$).

   Let $\dot Q$ be such a nice
   $\Por^*_\xi$-name.
   Then $\hat h^*_\xi(\dot Q)$ is also a
   nice $\Por^*_\xi$-name, and
   $\Vdash_{\Por^*_\xi} \varphi(\hat h^*_\xi(\dot Q),\xi)$ as $\varphi$ only uses ground model parameters (i.e., standard names).
\end{proof}

 As mentioned, we need closed
 $C\subseteq \pi^+$ that
 define complete subforcings.
 For this we use the following result:

\begin{lemma}\label{smallrestr}
    Let $\Por$ be as in Lemma~\ref{lem:kjgeqwtr},
    and let $\mu>\lambda$ be regular and $\aleph_1$-inaccessible.

    \begin{enumerate}[(a)]
        \item
        For $A\subseteq\pi^+$ of size ${<}\mu$
        there is some closed $C\supseteq A$
        of size ${<}\mu$, such that $\Por^*{\upharpoonright}C\lessdot\Por^*$.

        \item
        The
        closed sets $C\in [\pi^+]^{{<}\mu}$
        that satisfy $\Por^*{\upharpoonright}C\lessdot\Por^*$ form a $\lambda$-club.
    \end{enumerate}
\end{lemma}
The proof is just a standard Skolem-L\"owenheim type 
closure argument:
\begin{proof}

   (a) Using Corollary~\ref{cor:small}(\ref{item:quaxb})
   we can fix a function $f:(\Por^*_\pi)^2\to [\pi^+]^{<\lambda}$ such that
   if $p$ and $q$ are compatible
   then there is some $r\leq p,q$
   in $\Por^*{\restriction} f(p,q)$.
   Also, we can fix a function $g:(\Por^*_\pi)^{{\leq}\omega}\to[\pi^+]^{<\lambda}$ such that, whenever $\bar{p}=\la p_n:n<w\ra$ ($w\leq\omega$) is a non-empty antichain but \textbf{not}  maximal in $\Por_\pi$,
   then
   there is some
   $q\in \Por^*{\restriction} g(\bar{p})$
   with $q\perp p_n$ for any $n<w$. 
   By recursion on $\alpha<\lambda$, define
   \[A'_\alpha:=A\cup A_{{<}\alpha}\cup\bigcup_{p,q\in\Por^*{\upharpoonright}A_{{<}\alpha}}f(p,q)\cup\bigcup_{\bar{p}\in(\Por^*{\upharpoonright}A_{{<}\alpha})^{{\leq}\omega}} g(\bar{p})\]
   where $A_{{<}\alpha}:=\bigcup_{\xi<\alpha}A_\xi$; and let $A_\alpha:=\overline{A'_\alpha}$ be the
   closure (see Lemma~\ref{lem:closure}).
   So $|A_\alpha|<\mu$, and we can set $C:=\bigcup_{\alpha<\lambda}A_\alpha$, which is as desired because
   any countable sequence in $\Por^*{\upharpoonright}C$ is a countable sequence in $\Por^*{\upharpoonright}A_\alpha$ for some $\alpha<\lambda$.

   (b) Let $\la B_i : i\in\lambda\ra$ be an
   increasing sequence of closed subsets
   of $\pi^+$ such that $\Por^*{\upharpoonright}B_i$ is a complete
   subforcing of $\Por^*$.
   Set $B:=\bigcup_{i\in\lambda}B_i$.
   According to~\ref{lem:closure}(\ref{item:sjlgclosed})  $B$ is closed.
   Assume that $A\subseteq \Por^*{\upharpoonright}B$ is a maximal antichain. Any $p,q\in A$ are incompatible
   in $\Por^*{\upharpoonright}B_i$ for some
   $i$, and therefore in $\Por^*$.
   Due to ccc, $A$ is countable,
   and by Corollary~\ref{cor:small}(\ref{item:quaxb})
   there is an
   $i<\lambda$ such that $A\subseteq \Por^*{\upharpoonright}B_i$. Therefore $A$ is maximal in $\Por^*$.
\end{proof}


\begin{corollary}\label{smallrestr2}
   With the same hypothesis of Lemma~\ref{smallrestr}, if $P\subseteq\Por^*_\pi$ has size ${<}\mu$, then there is some $\Por^-\lessdot\Por^*_\pi$ of size~${<}\mu$ such that $P\subseteq\Por^-$.
\end{corollary}
\begin{proof}
   Apply Lemma~\ref{smallrestr} to $A:=\bigcup_{p\in P}H^*(p)$.
\end{proof}

We now summarize what we already know about the
construction that we are going to perform in
the next section:
\begin{corollary}\label{cor:summary}
   Let $\lambda$ be regular uncountable and assume
   $\lambda\le |\pi|$,
   $|\pi|^{<\lambda}=|\pi|$, and that
   $\pi\smallsetminus \pi_1$ is partitioned into $\Sigma$ and $\Pi$.
   We inductively construct a (tidy) S$\lambda$s iteration
   $\Por$ as follows:
   \begin{itemize}
       \item[$(\Sigma)$] 
       As step $\zeta\in \Sigma$,
       we pick
       a (definition of a) Suslin-ccc forcing
       $\Sor_\zeta$, some $\aleph_1$-inaccessible $\kappa_\zeta>\lambda$,
       and some $C_\zeta$
       in the $\lambda$-club set $[\zeta^+]^{<\kappa_\zeta}$ of Lemma~\ref{smallrestr}, and set $\Qor_\zeta=\Sor^{V^{\Por^*{\restriction} C_\zeta}}$. (So $\Por^-_\zeta=\Por^*{\upharpoonright}C_\zeta$.)

       \item[$(\Pi)$] 
       Fix a formula $\varphi(x,y)$ with parameters in
       the ground model. At step $\zeta\in\Pi$,
       pick some regular uncountable $\kappa_\zeta\le\lambda$
       and let
       $\Qnm_\zeta$ be (a suitable name for)
       the FS product of all Knaster
       real-number-posets of size ${<}\kappa_\zeta$
       satisfying $\varphi(x,\zeta)$.
   \end{itemize}
   Then (inductively)
   $\Por_\xi$ is a well defined ccc forcing for $\pi_1\leq\xi\leq\pi$, and
   \begin{enumerate}[(a)]
       \item $\LCU_{\textnormal{sp}}(\Por_\xi,\mu)$ holds for any regular $\lambda\le\mu\le\pi_0$.
       \item $\Por_\xi$ forces that the continuum
       has size below $|\pi|$.
   \end{enumerate}
\end{corollary}

\begin{proof}
    Each $\Qnm_\zeta$ is forced to be ccc (by either absoluteness or
   the Knaster assumption),
   so we get a valid iteration
   (and we assume that we choose the names for the iterands such that we get a tidy S$\lambda$s-iteration).

   $\LCU$ follows from Lemmas~\ref{lem:allthefactors} and \ref{lem:kjgeqwtr}(\ref{item:symmetric}), and
   Theorem~\ref{PresSplit}.

   For the size of the continuum, we use Corollary~\ref{cor:small}(\ref{item:quaxd}) to show by
   induction that $|\Por^*_\xi|\le|\pi|$:
   Assume this already
   is the case for
   $\zeta\in \Pi$,
   then $\Xi_\zeta$ consists of nice $\Por^*_\zeta$-names for ${<}\lambda$-sized real-number-forcings,
   and there are only
   $|\Por^*_\zeta|^{<\lambda}\le|\pi|$ many such nice names.
\end{proof}

\section{The forcing construction for the left hand side}\label{sec:left}

In this section, we prove the first step of the main theorem:\
Theorem~\ref{thm:step3}, which gives the independence results for the left hand side. After the work we have done in the previous sections, this is basically
a simple variant of the construction
in~\cite{GKS}.

\subsection{Preliminaries.}

\newcommand{\Aor}{\mathbb{A}}
\newcommand{\Bor}{\mathbb{B}}
\newcommand{\Dor}{\mathbb{D}}
\newcommand{\Eor}{\mathbb{E}}
\newcommand{\Lc}{\mathbbm{Lc}}


\begin{notation}\label{not:RLCU}
Denote
\[(\bfrak_i,\dfrak_i):=\left\{\begin{array}{ll}
(\addN,\cofN) & \text{if $i=1$,}\\
(\covN,\nonN) & \text{if $i=2$,}\\
(\bfrak,\dfrak) & \text{if $i=3$,}\\
(\nonM,\covM) & \text{if $i=4$,}\\
(\sfrak,\rfrak) & \text{if $i=\mathrm{sp}$.}
\end{array}\right.\]

As in~\cite{GKS,GKMS2}, for $i\in\{1,2,3,4,{\mathrm{sp}}\}$ consider Blass-uniform relational systems $\Rbf^{\LCU}_i$ and $\Rbf^{\COB}_i$ such that, following in Example~\ref{exp:blassunif}, $\Rbf^\LCU_{\mathrm{sp}}=\Rbf^\COB_{\mathrm{sp}}=\Rbsp$ and ZFC proves\footnote{In more detail, $\Rbf^{\LCU}_i=\Rbf^{\COB}_i$ except when $i=2$. If we follow~\cite{diegoetal} we can also consider $\Rbf^{\LCU}_2=\Rbf^{\COB}_2$.}
\[\bfrak(\Rbf^{\COB}_i)\leq\bfrak_i\leq\bfrak(\Rbf^{\LCU}_i)\text{\ and }\dfrak(\Rbf^{\LCU}_i)\leq\dfrak_i\leq\dfrak(\Rbf^{\COB}_i).\]
We abbreviate $\COB_{\Rbf^\COB_i}$ by $\COB_i$, and $\LCU_{\Rbf^\LCU_i}$ by $\LCU_i$.
\end{notation}

For completeness, we review the posets we use in our construction.

\begin{definition}
    Define the following forcing notions (where the forcing in item ($i$) is designed to increase $\bfrak_i$):
    \begin{enumerate}[(1)]

        \item \emph{Amoeba forcing $\Aor$} is the poset whose conditions are subtrees $T\subseteq 2^{<\omega}$ without maximal nodes such that $[T]$, the set of branches of $T$, has measure ${<}\frac{1}{2}$ (with respect to the Lebesgue measure of $2^\omega$). The order is $\supseteq$.

        \item \emph{Random forcing $\Bor$} is the poset whose conditions are subtrees $T\subseteq 2^{<\omega}$ without maximal nodes such that $[T]$ has positive measure. The order is $\subseteq$.

        \item \emph{Hechler forcing} is $\Dor:=\omega^{<\omega}\times\omega^\omega$ ordered by $(t,y)\leq(s,x)$ iff $s\subseteq t$, $x\leq y$ (pointwise) and $t(i)\geq x(i)$ for all $i\in|t|\smallsetminus|s|$.

        \item \emph{Eventually different forcing} is
        \[\Eor:=\omega^{<\omega}\times\bigcup_{n<\omega}\big([\omega]^{\leq n}\big)^\omega\]
        ordered by $(t,\psi)\leq(s,\varphi)$ iff $s\subseteq t$, $\forall i<\omega(\varphi(i)\subseteq\psi(i))$ and $t(i)\notin\varphi(i)$ for all $i\in|t|\smallsetminus |s|$.

        \item[(sp)] Let $F$ be a base of a (free) filter on $\omega$. \emph{Mathias-Prikry forcing on $F$} is
        $\Mor_F:=\{(s,x)\in[\omega]^{<\aleph_0}\times F: \max(s)<\min(x)\}$ (here $\max(\emptyset):=-1$) ordered by $(t,y)\leq(s,x)$ if $s\subseteq t$, $y\subseteq x$ and $t\smallsetminus s\subseteq x$.
    \end{enumerate}
\end{definition}

For each of the posets above it is easy to construct a 1-1 function from the poset into $\omega^\omega$. So, until the end of this section, the posets above are seen as subsets of $\omega^\omega$. Moreover, the posets (1)--(4) are Suslin ccc, and they are  homeomorphic to a Borel subset of $\omega^\omega$ (and the order is Borel as well).

In the proof of Theorem~\ref{mainstepI} we deal with special restrictions of the posets (1)--(4) under sets of reals of the following form.

\begin{definition}\label{def:elementary}
   Let $\lambda\geq\aleph_1$ be a cardinal. Say that $E\subseteq \omega^\omega$ is \emph{$\lambda$-elementary} if $E=\omega^\omega\cap N$ for some regular $\chi\geq(2^{\aleph_0})^+$ and some $N\preceq \mathcal H_\chi$ of size ${<}\lambda$, where $\mathcal H_\chi$ denotes the collection of hereditarily ${<}\lambda$-size sets.
\end{definition}

We look at posets of the form $\Sor\cap E$ where $\Sor$ is a poset as in (1)--(4) and $E\subseteq\omega^\omega$ is $\lambda$-elementary. Note that, whenever $\chi\geq(2^{\aleph_0})^+$, $N\preceq \mathcal H_\chi$ and $E=\omega^\omega\cap N$, we have $\Sor\cap E=\Sor^N$. Therefore:

\begin{fact}\label{fc:Suslinrestr}
   Let $E\subseteq\omega^\omega$ be elementary. Then:
   \begin{enumerate}[(1)]

       \item The poset $\Aor\cap E$ adds a (code of a) Borel measure zero set that contains all Borel null sets with Borel code in $E$.

       \item The generic real added by $\Bor\cap E$ evades all Borel null sets with Borel code in $E$.
       \item The generic real added by $\Dor\cap E$ dominates all the functions in $E$.

       \item The generic real added by $\Eor\cap E$ is eventually different from all the functions in $E$.
   \end{enumerate}
\end{fact}






We now show how to modify the forcing construction in~\cite[\S4 \& \S5]{GKMS1}
to include $\LCU_{\mathrm{sp}}$ and $\COB_{\mathrm{sp}}$, by performing a construction according to the previous section, in particular to
Corollary~\ref{cor:summary}.
We will assume the following:
%

\begin{assumption}\label{asm:bla}
    $k_0\in[2,\omega]$; $\lambda_\mfrak\leq\lambda_1\leq\lambda_2\leq\lambda_3<\lambda_4$ are uncountable regular cardinals, $\lambda_5\geq\lambda_4$ is a cardinal,
    $\lambda_3=\chi^+$,
    $\lambda_{\mfrak}\leq\lambda_{\mathrm{sp}}\leq\lambda_3$ regular, such that  $\chi^{<{\chi}}=\chi\geq\aleph_1$, $\lambda_5^{<\lambda_4}=\lambda_5$, and $\lambda_i$ is $\aleph_1$-inaccessible whenever $\lambda_i>\lambda_{\mathrm{sp}}$ and $1\leq i\leq 4$.
\end{assumption}

Our intention is to show the following:
\begin{goal}\label{mainstepI}
There is a ccc poset $\Por$ of size $\lambda_5$ such that, for any $i\in\{1,2,3,4,{\mathrm{sp}}\}$,
    \begin{enumerate}[(a)]
        \item $\LCU_i(\Por,\theta)$ holds for any regular $\lambda_i\leq\theta\leq\lambda_5$.
        \item There is some directed $S_i$ with $\cp(S_i)=\lambda_i$ and $|S_i|=\lambda_5$ such that $\COB_i(\Por,S_i)$ holds.
        \item $\Por$ forces $\pfrak=\sfrak=\lambda_{\mathrm{sp}}$ and $\cfrak=\lambda_5$.
        \item $\Por$ forces $\mfrak_k=\aleph_1$ for any $k\in[1,k_0)$, and $\mfrak_k=\lambda_\mfrak$ for any $k\in[k_0,\omega]$.\footnote{Note that $\lambda_\mfrak=\aleph_1$ is allowed.}
    \end{enumerate}
\end{goal}

The way to achieve this is parallel to
\cite[\S1]{GKS}:
As first step we give the ``basic construction'' in Lemma~\ref{lem:step1},
using ``simple bookkeeping'' (which is described
by parameters $\bar C=\la C_\alpha\ra_{\alpha\in\Sigma}$ in
the ground model). This gives
us everything apart from  $\LCU_3$
(i.e., we do not claim that $\bfrak$ remains small).
This first step contains the only new aspect of the construction:
As we use a variant of the construction according to
the previous section, we get
 $\LCU_\mathrm{sp}$.

The next steps are just as in
\cite[\S1.3 \& \S1.4]{GKS}.
In Lemma~\ref{lem:step2} we remark:
Assuming  $2^\chi\geq\lambda_5$ (in addition to
Assumption~\ref{asm:bla}),
we can choose the bookkeeping parameters
$\bar C$ in such a way that the resulting forcing
satisfies $\LCU_3$ and thus all of
Goal~\ref{mainstepI}.

And finally we show
Theorem~\ref{thm:step3}:
without the assumption $2^\chi\geq\lambda_5$ (while assuming~\ref{asm:bla})
we can also get all of
Goal~\ref{mainstepI}.
Why do we need to supress the assumption $2^\chi\geq\lambda_5$ from
Lemma~\ref{lem:step2}? Because we can then
additionally control the right-hand side characteristics
in Section~\ref{sec:15}, using the method of elementary submodels from~\cite{GKMS2}.


\medskip

In the following proof, we deal with the case $2\leq k_0<\omega$ and $\lambda_\mfrak>\aleph_1$.
In Section~\ref{ss:remaining} we mention the necessary changes for the remaining cases.

\medskip

\subsection{The basic forcing construction.}
   To each $1\leq i\leq 4$ associate a Suslin ccc poset as follows: $\Sor_1=\Aor$, $\Sor_2=\Bor$, $\Sor_3=\Dor$, and $\Sor_4=\Eor$.

   Set $\lambda:=\lambda_{\mathrm{sp}}$.
   Let $i^*$ be the minimal $i$ such that
   $\lambda_i>\lambda$. Note that $1\le i^*\le 4$.
   Set
   $I_1:=\{i^*,\dots, 4\}$ and
   $I_0:=\{\mlabel,\plabel\}\cup \{1,2,3,4\}\smallsetminus I_1$.

   Set $\pi_0:=\lambda_5$ (so $\pi_1=\omega_1\cdot\lambda_5$), and $\pi:=\pi_1+\lambda_5+\lambda_5$.
   Partition the final $\lambda_5$-interval of $\pi$, i.e. $\pi\smallsetminus (\pi_1+\lambda_5)$, into sets $\Pi_i$ ($i\in I_0$) and $\Sigma_i$ ($i\in I_1$), each of size $\lambda_5$.

   We construct a tidy S$\lambda$-s iteration, using
   $\Sigma:=\{\pi_1+\alpha:\, \alpha<\lambda_5\}\cup \bigcup_{i\in I_1} \Sigma_i$
   and $\Pi:=\bigcup_{i\in I_0} \Pi_i$.
%
   We will satisfy the requirements of Corollary~\ref{cor:summary},
   so in particular inductively we will have
   $|\Por^*_\xi|=\lambda_5$
   (and so ${\Por_\xi}\Vdash\cfrak=\lambda_5$) for all $\xi$.

   \begin{enumerate}[({I}1)]
     \item At stage
     $\zeta\in[\pi_1,\pi_1+\lambda_5)$ (in particular, $\zeta\in\Sigma$),
     we just add Cohen reals.
     More formally, to fit our framework,
     we set $\Sor_\zeta=\Cor=\omega^{<\omega}$ (Cohen forcing).
     Let
     $C_\zeta:=\emptyset$, which is closed and
     satisfies that $\Por^-_\zeta:=\Por^*_\zeta{\restriction} C_\zeta$ (i.e., the set containing only the empty condition) is a complete subforcing of $\Por^*_\zeta$. And
     $\Sor_\zeta^{V^{\Por^-_\zeta}}$ is Cohen forcing in the ground model,
     which is Cohen forcing in any extension by absoluteness.


     \item Assume $\zeta\in \Pi_i$ (for some $i\in I_0$).
         \begin{enumerate}[(i)]
           \item When $i=\mlabel$, let $\Xi_\zeta$ be the family of all nice $\Por^*_\zeta$-names of
           real-number-posets  of size ${<}\lambda_\mfrak$  that are forced (by $\Por^*_\zeta$) to be $k_0$-Knaster.


           \item When $i=\plabel$, let $\Xi_\zeta$ be the family of all nice $\Por^*_\zeta$-names of real-posets of size ${<}\lambda$ that are forced to be $\sigma$-centered.

           \item When $i\in\{1,2,3\}\cap I_0$,
           we consider $\Xi_\xi$ as the family of nice $\Por^*_\zeta$-names of all smaller-than-$\lambda_i$ versions of $\Sor_i$
           in the $\Por^*_\zeta$-extension,
           i.e., the
           forcings of the form
           \begin{equation*}
           Q=\Sor_i\cap E\text{ where $E$ is $\lambda_i$-elementary}
           \end{equation*}
           as in Definition~\ref{def:elementary}.
           %
           %
           %
           %
           Note that $\Sor_i$, and therefore
           also every variant $\Sor_i\cap E$,
           is linked and therefore Knaster.




         \end{enumerate}

     \item If $\zeta\in \Sigma_i$ (for some $i\in I_1$, so $\lambda_i>\lambda$),
     we pick (by suitable book-keeping) a  $C_\zeta\subseteq\zeta^+$ as in Lemma~\ref{smallrestr}(b). I.e.,
     $|C_\zeta| < \lambda_i$, $\Por^-_\zeta:=\Por^*{\upharpoonright}C_\zeta\lessdot\Por^*_\zeta$, and we set $\Sor_\zeta:=\Sor_i^{V^{\Por^-_\zeta}}$.
     (Here, suitable bookkeeping just means:
     For every
     $K\in [\pi^+]^{<\lambda_i}$
     there is some index $\zeta$
     such that $C_\zeta\supseteq K$.)

\end{enumerate}

We can now show that the construction
does what we want, apart from keeping $\bfrak$ small.

First let us note that sometimes it is more convenient to
view $\Por$ as a FS ccc iteration,
where we first add the
$\lambda_5$-many $\Gor_\Bbf$ forcings
(of size $\aleph_1$),
then the $\lambda_5$-many Cohen reals,
and then the rest of the iteration,
where we interpret each FS product
$\Qor_\zeta$ for $\zeta\in \Pi_i$
as a FS iteration with index set
$\lambda_5=|\Xi_i|$.
So all in all we can represent
$\Por$ as a FS iteration
\begin{equation}\label{eq:noprod}
\la P'_\alpha,\dot Q'_\alpha\ra_{\alpha\in \delta'}
\text{ of length }
\delta'=\lambda_5 + \lambda_5 + \Sigma_{\zeta\in \pi\smallsetminus (\lambda_5+\lambda_5)}\delta'_\zeta
\text{, with }
\delta'_\zeta:=\begin{cases}
\lambda_5 & \text{if } \zeta\in\Pi_i,\\
1 & \text{otherwise.}
\end{cases}
\end{equation}
For each $\alpha<\lambda_5+\lambda_5$,
$|\dot Q'_\alpha|\le\aleph_1$, and for each
$\alpha\ge \lambda_5+\lambda_5$ in $\delta'$, we say that $\dot Q'_\alpha$ ``is of type $i$'' for $i\in\{\textrm{m},\textrm{p},1,2,3,4\}$,
if either $\dot Q'_\alpha=\dot \Qor_\zeta$
for the respective $\zeta\in \Sigma_i$,
or if $\dot Q'_\alpha$ is a factor $\dot Q$
of $\Qnm_\zeta$
for the respective $\zeta\in \Pi_i$.
Note that $P'_{\lambda_5}=\Por_{\pi_1}$.

\begin{lemma}\label{lem:step1}
    The construction above satisfies Goal~\ref{mainstepI},
    apart from possibly $\LCU_3$.
\end{lemma}

\begin{proof}
\textbf{$\bm{\cfrak=\lambda_5}$}, as we use a construction
   following Corollary~\ref{cor:summary}.\smallskip

\textbf{Item (a) for $\bm{i=\mathrm{sp}}$}, i.e.,
   $\LCU_\textrm{sp}$: This also follows from Corollary~\ref{cor:summary}, and implies
   $\Vdash_{\Por} \sfrak\le \lambda$.\smallskip

\textbf{$\bm{\pfrak=\sfrak=\lambda}$}:
   To see $\pfrak\ge\lambda$ it is enough
   to show (in fact, equivalent, by Bell's theorem~\cite{MR643555}):
   For every $\sigma$-centered poset $Q'$
   of size ${<}\lambda$ (and contained in $\omega^\omega$), and any collection $D$ of size ${<}\lambda$ of dense subsets of $Q'$, there is a $Q'$-generic set over $D$.
   Any such $Q'$ and $D$ are forced to be already in
   the $\Por_\alpha$-extension for some $\alpha<\pi$.
   Pick some $\zeta\in\Pi_{\textrm{p}}$ larger than $\alpha$.
   Then a name $\dot Q$ for $Q'$
   is used as factor of $\Por_\zeta$,
   i.e., in $\Por_{\zeta+1}$ there is a $\dot Q$-generic
   object (over $D$).

   ZFC shows $\pfrak\le\sfrak$, and as $\sfrak\le\lambda$
   we get equality.\smallskip

   \textbf{Item (a) for $\bm{i\in\{1,2,4\}}$}, i.e.,
   $\LCU_i$: This is
   exactly as in~\cite[\S1.2]{GKS}.
   For this argument we interpret $\Por$ as the iteration $\la P'_\alpha,\dot Q'_\alpha\ra_{\alpha\in\delta'}$
   of~\eqref{eq:noprod}.
   However, we work in the
   $\Por_{\pi_1}$-extension (i.e., the $P'_{\lambda_5}$-extension).
   So we investigate the forcing which  first adds $\lambda_5$ many
   Cohens, and then a FS iteration of the iterands $Q'_\alpha$.

   As in~\cite[\S1.2]{GKS}, we now argue that each
   such $Q'_\alpha$ is  $(\Rbf^\LCU_i,\lambda_i)$-good.\footnote{
        The notion $(R,\theta)$-good was introduced by 
        Judah-Shelah~\cite{JS90} and Brendle~\cite{Bre91}, definitions can
        also be found in \cite[Def.~1.5]{GKS}, \cite[Def.~3.2]{GMS} or~\cite[Def. 6.4.4]{BaJu}.}
   So let us quickly check the cases (they are all summarized in~\cite[Lemma~1.6]{GKS}, and use results from \cite{JS90}, \cite{Kam89}, \cite{Bre91}).
   To get $(\Rbf^\LCU_1,\lambda_1)$-good:
 \begin{itemize}
     \item
   If $Q'_\alpha$ is of type $\mlabel$ or type $1$,
   then $Q'_\alpha$ has size ${<}\lambda_1$
   and thus is $(\Rbf^\LCU_i,\lambda_1)$-good (for any $i\in\{1,2,3,4\}$).
   \item
   If $Q'_\alpha$ is of type $\plabel$, $3$ or $4$,
   then $Q'_\alpha$ is $\sigma$-centered,
   and therefore $(\Rbf^\LCU_1,\aleph_1)$-good.
   \item
   If $Q'_\alpha$ is of type $2$, then
   it is a subalgebra
   of the measure algebra, and thus
   $(\Rbf^\LCU_1,\aleph_1)$-good.
  \end{itemize}
   For $(\Rbf^\LCU_2,\lambda_2)$-good the argument is even simpler:
   All $Q'_\alpha$ have size ${<}\lambda_2$ or are $\sigma$-centered;
   and for $(\Rbf^\LCU_4,\lambda_4)$-good the argument is trivial, as
   all $Q'_\alpha$ have size ${<}\lambda_4$.

   So this argument shows that, in the intermediate model
   $V^{\Por_{\pi_1}}$, the rest $P'$ of the forcing
   satisfies $\LCU_i(P',\lambda_i)$, witnessed by the Cohen
   reals $\{\eta_\alpha: \alpha\in [\pi_1,\pi_1+\lambda_i)\}$.
   This implies by definition of $\LCU$ that in the ground model
   $\LCU_i(\Por,\lambda_i)$ holds,
   witnessed by the same Cohen reals.\smallskip

   \textbf{Item (b) for $\bm{i\in I_1}$},
   i.e., $\COB_i$:
   This is also basically the same as in \cite[\S1.2]{GKS},
   where this time we argue from the ground model $V$,
   not the intermediate model $V^{\Por_{\pi_1}}$.
   We define the partial order
   $S_i$ to have domain $\Sigma_i$,
   ordered by $\zeta_1\le_{S_i} \zeta_2$
   iff $C_{\zeta_1}\subseteq C_{\zeta_2}$.

   Note that $C_\zeta$ is in
   $[\pi^+]^{<\lambda_i}$,
   $|\pi^+|=\lambda_5$,
   and our book-keeping ensures that
   $S_i$ is ${<}\lambda_i$-directed.
   Corollary~\ref{cor:small}(\ref{item:quaxc}) together with the fact
   that $\lambda\le \lambda_i$ shows that our bookkeeping will
   catch every real in the $\Por$-extension.
   Therefore $S_i$, and the generics added at stages in $S_i$, witness the COB property.\smallskip

   \textbf{Item (b) for $\bm{i\in \bm{I_0}\cap\{1,2,3\}}$}:
   This is very similar:
   Let $S_i$ be the set of pairs $(\zeta,\dot E)$ such that $\zeta\in\Pi_i$ and $\dot E$ is a nice $\Por^*_\zeta$-name of a $\lambda_i$-elementary subset of $\omega^\omega$.
   We order  $S_i$ as follows:  $(\xi_1,\dot E_1)\le_i(\xi_2,\dot E_2)$
   iff $\xi_1\leq\xi_2$ and the empty condition forces
   that $\dot E_1\subseteq
   \dot E_2$.

   For $(\zeta,\dot E)\in S_i$, $\Sor_i\cap\dot E$ forms part of the FS product $\Qnm_\zeta$, so $\Por_{\zeta+1}$ adds a $\Sor_i\cap\dot E$-generic object $\dot{y}_{\zeta,\dot E}$ as in Fact~\ref{fc:Suslinrestr}. We show that $S_i$ and $\{\dot{y}_{\zeta,\dot E}:\, (\zeta,\dot E)\in S_i\}$ witness $\COB_i$.


   Let $\dot r$ be a $\Por^*$-name of a real,
   then $\dot r$ is a $\Por^*_{\xi_0}$-name of a real
   for some $\xi_0<\pi$, and there is some $\dot E_0$
   such that $(\xi_0,\dot E_0)\in S_i$ and $\Vdash_{\Por_{\xi_0}}\dot r\in\dot E_0$.
   Hence, whenever $(\xi,\dot E)\in S_i$ is above $(\xi_0,\dot E_0)$,
   $\Vdash_\Por \dot r\in\dot E$ so $\dot y_{\xi,\dot E}$ is generic over $\dot r$.


   And for any ${<}\lambda_i$-sequence $\la E_j:j\in J\ra$
   of nice names for
   $\lambda_i$-elementary sets $E_j$
   we can find a
   nice name for a $\lambda_i$-elementary set
   $E\supseteq \bigcup_{j\in J} E_j$.
   This shows that $S_i$ is
   ${<}\lambda_i$-directed.\smallskip

   \textbf{Item (b) for  $\bm{i=\mathrm{sp}}$,} i.e.,
   $\COB_\textrm{sp}$:
   This is
   basically the same:
   Among the $\sigma$-centered
   forcings that we use as factors
   in step $\zeta$ of type $\plabel$, there are
   Mathias-Prikry forcings $\Mor_{\dot F}$ on (free) filter bases of size ${<}\lambda$.
   In more detail: Assume $\dot F$ is a $\Por^*_\zeta$-name for a filter base of size ${<}\lambda$,
   so set $\dot Q:=\Mor_{\dot F}$.
   Then $\dot Q$ is $\sigma$-centered and
   adds a real which is not split by any set in
   $\dot F$.

   So let $S_{\mathrm{sp}}$ be the set of pairs $(\zeta,\dot F)$ such that $\zeta\in\Pi_{\plabel}$ and $\dot F$ is a nice $\Por^*_\zeta$-name of a filter base of size ${<}\lambda$. Set $(\xi_1,\dot F_1)\le_{\plabel}(\xi_2,\dot F_2)$
   iff $\xi_1\leq\xi_2$ and the empty condition forces
   that $\dot F_1\subseteq
   \dot F_2\cup\dot F^d_2$, where $F^d:=\{\omega\menos x:x\in F\}$.
   For $(\xi,\dot F)\in S_{\plabel}$, let $\dot{y}_{\xi,\dot F}$ be the $\Por_{\xi+1}$-name of the
   generic real added by $\Mor_{\dot F}$. It follows that $S_{\plabel}$ and $\{\dot{y}_{\xi,\dot F}:\, (\xi,\dot F)\in S_{\plabel}\}$ witness $\COB_{\mathrm{sp}}$.

   \smallskip
   %
   %
   %
   %

   \textbf{Item (d)} is exactly the same as in~\cite[Lemma~4.7]{GKMS1}.
\end{proof}

To guarantee $\bfrak\le\lambda_3$, we have to make sure that the large iterands
(i.e., the forcings of size $\ge\lambda_3$)
do not destroy $\LCU_3$ (small forcings are, as usual, harmless). In our construction,
the only large forcings
are the partial
eventually different forcings at steps $\zeta\in\Sigma_4$.
For these forcings, we introduce in~\cite{GKS}
(based on~\cite{GMS}) ultrafilter-limits
and use them to preserve $\LCU_3$. The same argument works here.

\begin{remark}
   Note that in the proof of Lemma~\ref{lem:step1} we do not require the hypotheses $\chi=\chi^{{<}\chi}$ and $\lambda_3=\chi^+$ from Assumption~\ref{asm:bla}. These will be used to guarantee $\LCU_3$ in the following subsection.

   If in Assumption~\ref{asm:bla} we consider $\lambda_3=\lambda_4$ (instead of $\lambda_3<\lambda_4$), then the same proof of Lemma~\ref{lem:step1} guarantees Goal~\ref{mainstepI} in full (i.e., including $\LCU_3$). When $\lambda_{\mathrm{sp}}=\lambda_4$, in the forcing construction above we have $\Sigma=[\pi_1,\pi_1+\lambda_5)$.
\end{remark}

\subsection{Dealing with~\texorpdfstring{$\bfrak$}{b}} 

\begin{lemma}\label{lem:step2}
  In addition to Assumption~\ref{asm:bla}
  we suppose that $2^\chi\geq\lambda_5$.
  Then we can choose $C_\zeta$
  for all $\zeta\in \Sigma_4$
  such that
  $\LCU_3(\Por,\kappa)$ holds
  for all regular $\kappa\in[\lambda_3,\lambda_5]$.
  Moreover, in the inductive construction,
  for each $\zeta\in \Sigma_4$ there is a
  $\lambda$-club of $[\zeta^+]^{<\lambda_4}$
  such that we can choose $C_\zeta$ from this club set.
\end{lemma}

\begin{proof}
   This is analogous to~\cite[\S1.3]{GKS}, in particular
   to Lemma/Construction 1.30. We will only remark
   on the required changes.
   Again we interpret $\Por$
   as in~\eqref{eq:noprod}.

   We work from the ground model, not in the intermediate
   $\Por_{\pi_1}$-extension.
   Accordingly, we have to
   incorporate the initial segment  of the iteration
   $\Por_{\pi_1}=P'_{\lambda_5}$ into the
   argument. This is no problem, as we just have to deal with another type of
   small forcing, the $\dot Q'_\alpha$ for $\alpha<\lambda_5$, which all have size $\aleph_1$.

   Of course, $E':=\Eor^{V^{\Por^*_\zeta{\restriction} C_\zeta}}$ is closed under conjunctions of conditions, i.e., satisfies the assumptions of \cite[Fact~1.25]{GKS}.
   And instead of ``ground model code sequences''
   we use ``nice $\Por^*_\zeta{\restriction} C_\zeta$-names''.


   The crucial part of the old proof is \cite[Lemma~1.30(d)]{GKS}.
   There, we use the notation $w_\alpha\subseteq \alpha$,
   and $Q_\alpha$ are those $\Eor$-conditions
   that can be calculated in a Borel way from the generics with indices in $w_\alpha$, i.e., $Q_\alpha=\Eor\cap V^{\Por^*{\restriction}w_\alpha}$;
    and we show that the set of ``suitable'' $w_\alpha$
    is an $\omega_1$-club in $[\alpha]^{<\lambda_4}$, where ``suitable'' means:
    If we have a ground-model-sequence of (nice) $Q_\alpha$-names,
    then the $D_\alpha^\varepsilon$-limit (a well-defined condition in eventually different forcing)
    is also element of $Q_\alpha$ (for all $\varepsilon\in\chi)$.

    The same argument gives us the following
    for our new framework:
    We can
    perform the construction of Lemma~\ref{lem:step1}
    and, at all indices $\zeta$ of type $4$,
    the set of ``suitable'' $C_\zeta\in [\zeta^+]^{<\lambda_4}$
    is a $\lambda$-club, where suitable now means the following
    (recall that
we have $\Qnm_\zeta=\Eor^{\Por^*_\zeta{\restriction} C_\zeta}$):
    For any sequence of nice
    $\Por^*_\zeta{\restriction} C_\zeta$-names for elements
    of $\Eor$,\footnote{Note: as $|C_\zeta|<\lambda_4$,
    and $\lambda_4$ is $\aleph_1$-inaccessible,
    there  are ${<}\lambda_4$ many such sequences, cf.~Lemma~\ref{smallH} (and~\ref{item:restr}).}
    the $D_\alpha^\varepsilon$-limit of this sequence is forced to be in
    $\Qnm_\zeta$ as well.

    Here, we only get a $\lambda$-club
    and not an $\omega_1$-club,
    as only for increasing unions
    of length $\lambda$
    we have
    \[
      \bigcup_{i\in\lambda}
      \bigl(\Por^*_\zeta{\restriction} C_i
      \bigr)
      =
      \Por^*_\zeta{\restriction}\bigg( \bigcup_{i\in\lambda} C_i\bigg).
    \]
    Also, we now have to choose
    $C_\zeta$ not only in this
    $\lambda$-club, but in the intersection
    with the $\lambda$-club of Lemma~\ref{smallrestr}(b)
    (so that we get a closed $C_\zeta$
    such that $\Por^*_\zeta{\restriction} C_\zeta\lessdot \Por^*_\zeta$ as required
    for our construction.)

    The same argument as in the old proof (Lemma 1.31 there) then shows:
    Whenever all $C_\zeta$ as chosen ``suitably'' (for all $\zeta$ of type $4$),
    we get $\LCU_3$.
\end{proof}

\begin{theorem}\label{thm:step3}
  Assumption~\ref{asm:bla} is enough to
  find a $\Por$ as required for Goal~\ref{mainstepI}.
\end{theorem}

\begin{proof}
Let $R$ be the poset of partial functions $r:\chi\times\lambda_5\to\{0,1\}$ with domain of size ${<}\chi$ (ordered by extension).
As we assume $\chi^{<\chi}=\chi$, this poset is $\chi^+$-cc, and obviously
${<}\chi$-closed, so it does not change any cofinalities.
As in the old proof, at each step $\zeta$ of type $4$
in the inductive construction of $\Por$, we can go into
the $R$-extension of the ground model, use Lemma~\ref{lem:step2}
to get a suitable $C_\zeta^0$ (above some initial
set given by the usual book-keeping), find
in $V$ some $\tilde C_\zeta^0$ such that
$C_\zeta^0$ is forced to be a subset. Now we iterate
this $\lambda$ many times (not just $\omega_1$ as in the old proof), taking unions at limits,
and use the fact that the ``suitable''
parameters $C_\zeta$ are closed under $\lambda$-unions (they form a $\lambda$-club in $[\zeta^+]^{<\lambda_4}$).

This way we get a sequence of
parameters $C_\zeta$ \emph{in the ground model},
such that if we define in the $R$-extension a forcing  $\Por'$
using these parameters we get $\LCU_3(\Por',\kappa)$; a simple
absoluteness argument~\cite[Lemma~1.33]{GKS} then shows that these parameters will
already define in $V$ a forcing $\Por$ with  $\LCU_3(\Por,\kappa)$.

Note:
We do not interpret $\Xi_\zeta$ (for $\zeta\in\Pi$) in the $R$-extension, but use it with the same meaning it has in $V$. So $\Por'$ may not be symmetric in the $R$-extension, but this is not important here:
We are only interested in $\LCU_3(\Por',\kappa)$ in this argument, and we
do not claim that $\Por'$ in the $R$-extension satisfies the other
properties we have already shown for $\Por$. And for $\LCU_3(\Por',\kappa)$,
any iterand that has size ${<}\lambda_3$ is unproblematic.
\end{proof}

\begin{remark}\label{rem:limitb}
   It is not necessary to restrict $\lambda_3$ to a successor cardinal in Assumption~\ref{asm:bla}. To allow regular $\lambda_3$ in general, we forget about $\chi$ in Assumption~\ref{asm:bla} and just assume that $\lambda_3^{<\lambda_3}=\lambda_3>\aleph_1$. In this way, Lemma~\ref{lem:step2} is valid by assuming $2^{\lambda_3}\geq\lambda_5$ instead, and Theorem~\ref{thm:step3} is true when replacing $\chi$ by $\lambda_3$ in the proof (i.e., $R$ gets modified and it forces $2^{\lambda_3}\geq\lambda_5$). No further changes in the proofs (even in those from~\cite{GKS}) are needed to justify this.

   On the other hand, can we allow $\lambda_3=\aleph_1$ in Assumption~\ref{asm:bla}? (So all cardinals except $\lambda_4$ and $\lambda_5$ are $\aleph_1$.) Although we can make the construction in this case, now the forcings $\Gor_{\Bbf_\delta}$ have size $\lambda_3=\aleph_1$, so they could destroy $\LCU_3(\Por,\aleph_1)$. An alternative to deal with this problem is to perform a similar iteration with $\pi_0=0$ (so $\pi_1=0$, that is, no initial FS product of $\Gor_{\Bbf}$ is used) and guarantee $\LCU_{\Rbf^*}(\Por,\kappa)$ for any regular $\kappa\in[\aleph_1,\lambda_5]$ with the methods of this subsection (i.e.\ the methods from~\cite[\S1.3 \& \S1.4]{GKS}) adapted to $\Rbf^*$, where $\Rbf^*$ is the Blass-uniform relational system from~\cite{Ksplit} (see also~\cite[Example~2.19]{MejMod}) such that $\bfrak(\Rbf^*)=\max\{\bfrak,\sfrak\}$ and $\dfrak(\Rbf^*)=\min\{\dfrak,\rfrak\}$.
\end{remark}

\subsection{The other constellations for the Knaster numbers}\label{ss:remaining}
So far we have assumed that
$\lambda_\mfrak>\aleph_1$ and that
$k_0<\omega$. We now remark on how to prove the
other cases:\smallskip

\emph{Case $\lambda_\mfrak=\aleph_1$.} We only change $I_0:=\{\plabel\}\cup\{i\in[1,4]:\lambda_i\leq\lambda\}$
   (so (I2)(i) is excluded in the construction). Check details in~\cite[Lemma~4.7]{GKMS1}. Note that here the value of $k_0$ is irrelevant.\smallskip

\emph{Case $k_0=\omega$ and $\lambda_\mfrak>\aleph_1$.} Force with $P_{\mathrm{cal},\lambda_\mfrak}\ast\Por$ where $P_{\mathrm{cal},\lambda_\mfrak}$ is the precaliber $\aleph_1$
   poset from~\cite[\S5]{GKMS1} and $\Por$ is the
   forcing resulting from the construction above (in the $P_{\mathrm{cal},\lambda_\mfrak}$-extension).\footnote{For $i=\mlabel$, recall that ``$\omega$-Knaster" abbreviates ``precaliber $\aleph_1$".}


\subsection{The alternative order of the left side}

The construction of~\cite[\S2]{KeShTa:1131} for the alternative order of the left side of Cicho\'n's diagram can also be adapted in the situation of the previous theorems. This is just interchanging the order of the values of $\bfrak$ and $\covN$, that is, instead of forcing $\covN=\lambda_2\leq\bfrak=\lambda_3$, we force $\bfrak=\lambda_3<\covN=\lambda_2$. See also \cite{modKST} for the weakening of the hypothesis GCH:

\begin{theorem}\label{altmainstepI}
    Theorem~\ref{thm:step3} (and Goal~\ref{mainstepI}) is still valid when, in Assumption~\ref{asm:bla}, we replace $\lambda_1\leq\lambda_2\leq\lambda_3<\lambda_4$ by $\lambda_1\leq\lambda_3<\lambda_2<\lambda_4$.\footnote{As in~\cite[\S2]{KeShTa:1131}, the relational system $\Rbf^\LCU_2$ corresponding to this result is not the same as the one for Theorem~\ref{thm:step3}. Although this is a relational system of the reals, it is not Blass-uniform.}
\end{theorem}

Remark~\ref{rem:limitb} also applies in this situation.

\section{15 values}\label{sec:15}

In this section, we review some tools from~\cite{GKMS2,GKMS1} and show how they are used to control the cardinal characteristics other than $\sfrak$.  We describe the forcing constructions but we omit the details in the proofs, since these are exactly as in the cited references.

We use the notions of \emph{\mlike\ cardinal characteristic} and \emph{\hlike\ characteristic} from~\cite[\S3]{GKMS1}. We do not need to recall their definition, but we only need some of their properties and to know that the cardinals $\mfrak_k$ ($1\leq k\leq \omega$) are \mlike, $\hfrak$ and $\gfrak$ are \hlike, and $\pfrak$ and $\tfrak$ are of both types.

\begin{lemma}[{\cite[Cor.~3.5]{GKMS1}}]\label{lem:mlike}
   Let $\kappa$ be an uncountable regular cardinal, $\lambda$ a cardinal, $\xfrak$ a cardinal characteristic, and let $\Por$ be a $\kappa$-cc poset that forces $\xfrak=\lambda$ (so $\lambda$ is a cardinal in the $\Por$ extension). If $M\preceq \mathcal H_\chi$ (with $\chi$ a large enough regular cardinal) is ${<}\kappa$ closed and contains (as elements) $\Por,\kappa,\lambda$ and the parameters of the definition of $\xfrak$, then $\Por\cap M$ is a complete subposet of $\Por$ and:
   \begin{enumerate}[(i)]
       \item If $\xfrak$ is \mlike\ and $\lambda\geq\kappa$,
       then $\Por\cap M\Vdash \xfrak\geq\kappa$.
       \item If $\xfrak$ is \mlike\ and $\lambda<\kappa$,
       then $\Por\cap M\Vdash \xfrak=\lambda$.
       \item If $\xfrak$ is \hlike,
       then $\Por\cap M\Vdash\xfrak\leq|\lambda\cap M|$.
   \end{enumerate}
\end{lemma}

\begin{lemma}[{\cite[Lemma~6.3]{GKMS1}}]\label{lem:gfrak}
Assume:
\begin{enumerate}[(1)]
\item $\kappa\le\nu$ are uncountable regular
cardinals,
$\Por$ is a $\kappa$-cc poset.

\item $\mu=\mu^{<\kappa}\geq\nu$ and $\Por$ forces $\cfrak>\mu$.

\item For some relational systems of the reals $\Rbf^1_i$ ($i\in I_1$) and some regular $\lambda^1_i\leq\mu$:
$\Por$ forces $\LCU_{\Rbf^1_i}(\lambda^1_i)$

\item For some relational systems of the reals $\Rbf^2_i$ ($i\in I_2$), and some directed order $S^2_i$ with $\bfrak(S^2_i)=\lambda^2_i\leq\mu$ and $|S^2_i|\leq\vartheta^2_i\leq\mu$: $\Por$ forces  $\COB_{\Rbf^2_i}(S^2_i)$.

\item For some \mlike\ characteristics $\mathfrak y_j$ ($j\in J$)
and $\lambda_j<\kappa$:
$\Vdash_\Por \mathfrak y_j=\lambda_j$.

\item For some \mlike\ characteristics $\mathfrak y'_k$ ($k\in K$): $\Vdash_\Por \mathfrak{y}'_k\geq\kappa$.

\item $|I_1\cup I_2\cup J\cup K|\leq\mu$.
\end{enumerate}
Then there is a complete subforcing $\Por^*$ of $\Por$ of size $\mu$ forcing:
\begin{enumerate}[(a)]
    \item $\mathfrak y_j=\lambda_j$, $\mathfrak{y}'_k\geq\kappa$, $\LCU_{R^1_i}(\lambda^1_i)$ and $\COB_{R^2_{i'}}(\lambda^2_{i'},\vartheta^2_{i'})$ for all $i\in I_1$, $i'\in I_2$, $j\in J$ and $k\in K$;
    \item $\cfrak=\mu$ and $\gfrak\leq\nu$.
\end{enumerate}
\end{lemma}

We are now ready to prove the main result of this paper. We use Notation~\ref{not:RLCU} and the following assumption for all the results in this section.

\begin{assumption}\label{hyp:11}
   \
   \begin{enumerate}[(1)]
       \item $\mu_\mfrak\leq\mu_\pfrak\leq\mu_{0}\leq\mu_1\leq\mu_2\leq\ldots\leq\mu_8$ are uncountable regular.
       \item $\mu_9\geq\mu_8$ is a cardinal such that 
       $\mu_9^{{<}\mu_{0}}=\mu_9$.

       \item\label{newcard} $0\leq i_0\leq2$, $\mu_\mathrm{sp}\in[\mu_{i_0},\mu_{i_0+1}]$ and $\mu_\rfrak\in[\mu_{8-i_0},\mu_{9-i_0}]$ are regular.

       \item\label{11card}
       There are eleven regular cardinals $\theta_0>\cdots> \theta_{10}>\mu_9$ such that $\theta_i^{{<}\theta_i}=\theta_i$ for any $i<11$, $\theta_i$ is $\aleph_1$-inaccessible for $i\in\{1,3,5,7\}$, $\theta_3=\chi_3^+$ and $\chi_3=\chi_3^{<\chi_3}$.\footnote{%
       %
       We could further weaken the assumption depending on the value $i_0$. E.g.,
       in  case $i_0=1$, $\theta_i$ is required   $\aleph_1$-inaccessible only for $i\in\{1,3,5\}$. Also, it is enough that $\theta_0^{{<}\theta_1}=\theta_0$ (here, $\theta_0$ could be singular), and $\theta_3$ is not needed successor according to Remark~\ref{rem:limitb}. For more pedantic weakenings, see~\cite[Rem.~3.5]{GKMS2}.%
       %
       }

   \end{enumerate}
\end{assumption}

Note that, under GCH,
assumption (\ref{11card}) is irrelevant, and
$\mu_9^{{<}\mu_{0}}=\mu_9$ is equivalent to $\cof(\mu_9)\geq\mu_{0}$.

The Main Theorem for Figure~\ref{fig:cichonorders}(A) is proved in two steps through the following two results.

\begin{theorem}\label{mainstep2}
   Under Assumption~\ref{hyp:11}, for any $k_0\in[2,\omega]$ there is a ccc poset $\Por^1$ such that, for any $i\in\{1,2,3,4,\mathrm{sp}\}$,
    \begin{enumerate}[(a)]
        \item $\LCU_i(\Por^1,\theta)$ holds for $\theta\in\{\mu_i,\mu_{9-i}\}$, where $\mu_{\mathrm{sp}}:=\mu_\sfrak$ and $\mu_{9-\mathrm{sp}}:=\mu_\rfrak$.
        \item There is some directed $S_i$ with $\cp(S_i)=\mu_i$ and $\cf(S_i)=\mu_{9-i}$ such that $\COB_i(\Por^1,S_i)$ holds.
        \item $\Por^1$ forces $\pfrak=\gfrak=\mu_0$ and $\cfrak=\mu_9$.
        \item $\Por^1$ forces $\mfrak_k=\aleph_1$ for any $k\in[1,k_0)$, and $\mfrak_k=\mu_\mfrak$ for any $k\in[k_0,\omega]$.
    \end{enumerate}
\end{theorem}
\begin{proof}
    We deal with the case $i_0=1$, that is, $\mu_1\leq\mu_\mathrm{sp}\leq\mu_2$ (any other case is similar). We rewrite the sequence
    \[
    \xymatrix@=0.1ex{
         \mu_1&\leq&\mu_\mathrm{sp}&\leq&\mu_2&\leq&
         \mu_3&\leq&\mu_4&\leq&\mu_5&\leq&
         \mu_6&\leq&\mu_{\rfrak}&\leq&\mu_7&\leq&\mu_8&\leq&\mu_9&\text{as}  \\
         \vartheta_{10}&\leq&\vartheta_8&\leq&\vartheta_6&\leq&\vartheta_4&\leq&\vartheta_2&\leq&\vartheta_1&\leq&\vartheta_3&\leq&\vartheta_5&\leq&\vartheta_7&\leq&\vartheta_9&\leq&\vartheta_{11},
    }
    \]
    and let $\la\theta_j:j<11\ra$ be cardinals as in Assumption~\ref{hyp:11}(\ref{11card}) ordered by
    \[\vartheta_{11}<\theta_{10}<\theta_9<\cdots<\theta_0\]
    as shown in Figure~\ref{fig:setup2}.

    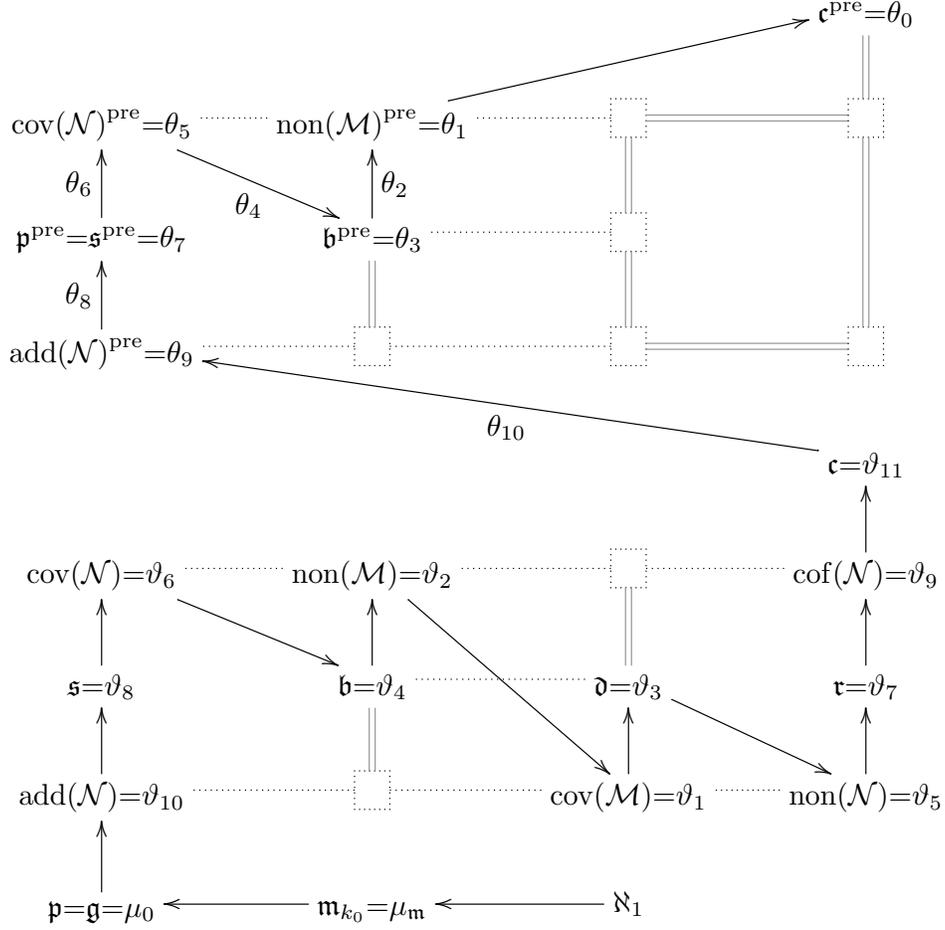
\begin{figure}
\resizebox{\textwidth}{!}{$
\xymatrix@=5ex{
&
&
&\txt{$\cfraki{=}\theta_0$}
\\
\txt{$\covNi{=}\theta_5$}
\ar@{.}[r]\ar[dr]_-{\textstyle\theta_4}
&\txt{$\nonMi{=}\theta_1 $}
\ar@{.}[r]\ar[rru]
&\mye
\ar@{=}@[Gray][r]
&\mye
\ar@{=}@[Gray][u]
\\
\txt{$\pfrak^{\mathrm{pre}}{=}\sfrak^{\mathrm{pre}}{=}\theta_7$}
\ar[u]^-{\textstyle\theta_6}
&\txt{$\bfraki{=}\theta_3$}
\ar@{.}[r]\ar[u]_-{\textstyle\theta_2}
&\mye
\ar@{=}@[Gray][u]
\\
\txt{$\addNi{=}\theta_9$}
\ar@{.}[r]\ar[u]^-{\textstyle\theta_8}
&\mye
\ar@{.}[r]\ar@{=}@[Gray][u]
&\mye
\ar@{=}@[Gray][r]\ar@{=}@[Gray][u]
&\mye
\ar@{=}@[Gray][uu]
\\
&
&
&\txt{$\cfrak{=}\vartheta_{11}$}
\ar[ulll]^-{\textstyle\theta_{10}}
\\
\txt{$\covN{=}\vartheta_6$}
\ar@{.}[r]\ar[rd]
&\txt{$\nonM{=}\vartheta_2$}
\ar@{.}[r]\ar[ddr]
&\mye
\ar@{.}[r]
&\txt{$\cofN{=}\vartheta_{9}$}
\ar[u]
\\
\txt{$\sfrak{=}\vartheta_8$}
\ar[u]
&\txt{$\bfrak{=}\vartheta_4$}
\ar@{.}[r]\ar[u]
&\txt{$\dfrak{=}\vartheta_3$}
\ar@{=}@[Gray][u]\ar[dr]
&\txt{$\rfrak{=}\vartheta_7$}
\ar[u]
\\
\txt{$\addN{=}\vartheta_{10}$}
\ar@{.}[r]\ar[u]
&\mye\ar@{.}[r]
\ar@{=}@[Gray][u]
&\txt{$\covM{=}\vartheta_1$}
\ar@{.}[r]\ar[u]
&\txt{$\nonN{=}\vartheta_5$}
\ar[u]
\\
\txt{$\pfrak{=}\gfrak{=}\mu_0$}
\ar[u]
&\txt{$\mfrak_{k_0}{=}\mu_{\mfrak}$}
\ar[l]
&\aleph_1
\ar[l]
}
$}
\caption{The cardinals $\vartheta_n$ and $\theta_n$ are increasing along the arrows.
The upper diagram shows the situation forced by $\Por^0$,
and the lower diagram shows the one forced by $\Por^1$. ($\sfrak$ can be anywhere between $\pfrak$ and $\bfrak$.)}\label{fig:setup2}
\end{figure}

    Let $\Por^0$ be the ccc poset obtained by application of Theorem~\ref{thm:step3} to $\lambda_\mfrak=\mu_\mfrak$, $\lambda_1=\theta_9$, $\lambda_{\mathrm{sp}}=\theta_7$, $\lambda_2=\theta_5$, $\lambda_3=\theta_3$, $\lambda_4=\theta_1$ and $\lambda_5=\theta_0$. In particular, this forces the top diagram of Figure~\ref{fig:setup2} and item (d). We show how to construct a complete subforcing of $\Por^0$ that satisfies the statement of the theorem, in particular, it forces the bottom diagram of Figure~\ref{fig:setup2}.

    For $1\leq n\leq 10$ and $\alpha<\vartheta_n$ define $M_{n,\alpha}$ fulfilling:
    \begin{itemize}
        \item $M_{n,\alpha}\preceq \mathcal H_\chi$ (for a fixed large enough regular $\chi$) and it contains (as elements) the sequences of $\theta$'s and $\vartheta$'s, $\Por^0$ and the directed sets associated with the $\COB$ properties forced by $\Por^0$.
        \item The sequences $\la M_{m,\xi} : \xi<\vartheta_m \ra$ for $1\leq m<n$ and $\la M_{n,\xi}:\xi<\alpha\ra$ belong to $M_{n,\alpha}$.
        \item $M_{n,\alpha}$ is ${<}\theta_n$ closed of size $\theta_n$.
    \end{itemize}
    Set $M_n:=\bigcup_{\alpha<\vartheta_n}M_{n,\alpha}$ and $M^+:=\bigcap_{n=1}^{10} M_n$. Exactly as in the proof of~\cite[Thm.~3.1]{GKMS2} one can show that  $M^+\preceq \mathcal H_\chi$, $M^+$ is ${<}\vartheta_{10}$-closed, and $\Por':=\Por^0\cap M^+$ is a ccc poset that forces (a), (b) and $\cfrak=\theta_{10}$. Even more, $\Por'$ forces (d) and $\pfrak\geq\vartheta_{10}$ by Lemma~\ref{lem:mlike}.

    The desired poset is a complete subposet $\Por_1$ of $\Por'$ of size $\vartheta_{11}$ obtained by direct application of Lemma~\ref{lem:gfrak} (to $\kappa=\nu=\mu_0$ and $\mu=\vartheta_{11}$).
\end{proof}

\begin{theorem}\label{mainfinal}
   Under Assumption~\ref{hyp:11}, for any $k_0\in[2,\omega]$ there is a cofinality preserving poset $\Por$ such that, for any $i\in\{1,2,3,4,\mathrm{sp}\}$, it satisfies (a), (b) and (d), and $\Por$ forces $\pfrak=\mu_\pfrak$, $\hfrak=\gfrak=\mu_0$ and $\cfrak=\mu_9$.
\end{theorem}
\begin{proof}
   Let $\Qor:=\mu_\pfrak^{{<\mu_\pfrak}}$ ordered by end extension, and let $\Por^1$ be the poset constructed in Theorem~\ref{mainstep2}. Exactly as in the proof of~\cite[Thm.~7.4]{GKMS1}, $\Por:=\Por^1\times\Qor$ is as required.
\end{proof}

In the same way, we can prove the Main Theorem corresponding to Figure~\ref{fig:cichonorders}(B). In this case, we initial forcing $\Por^0$ is obtained from Theorem~\ref{altmainstepI}.

\begin{theorem}\label{altmainfinal}
   Both Theorems~\ref{mainstep2} and~\ref{mainfinal} are valid when Assumption~\ref{hyp:11} is modified in the following way:
   \begin{enumerate}[(i)]
       \item We replace the order of the regular cardinals in (1) by
   \[\mu_\mfrak\leq\mu_\pfrak\leq\mu_0\leq\mu_1\leq\mu_3\leq\mu_2\leq\mu_4\leq\mu_5\leq\mu_7\leq\mu_6\leq\mu_8.\]
       \item In (\ref{newcard}), we consider $i_0\in\{0,1\}$, but $\mu_{\mathrm{sp}}\in[\mu_1,\mu_3]$ and $\mu_\rfrak\in[\mu_6,\mu_8]$ when $i_0=1$.
       \item In (\ref{11card}), 
       instead of $\theta_3=\chi_3^+$ and $\chi_3^{<\chi_3}=\chi_3$, assume $\theta_5=\chi_5^+$
       and $\chi_5^{<\chi_5}=\chi_5$.
   \end{enumerate}
\end{theorem}

\section{Discussions}\label{sec:disc}

One obvious question is:
\begin{question}
   How to separate additional
   cardinals from Figure~\ref{fig:all20}?
\end{question}

Another one:
\begin{question}
   How to get other orderings, where
   $\nonM>\covM$?
\end{question}
This is not possible with
FS ccc iterations, as any such iteration whose length has uncountable cofinality $\delta$ forces $\nonM\le\cof(\delta)\le\covM$, so alternative methods are required.
A creature forcing method based on the notion of decisiveness~\cite{MR2499421,MR2864397} has been developed in~\cite{MR3696076} to separate five characteristics in Cich\'on's diagram, but this method is restricted to $\omega^\omega$-bounding forcings, i.e., results in $\dfrak=\omega_1$. An unbounded decisive creature construction might be promising. Alternatively, Brendle proposed a method of shattered iterations~\footnote{J.\ Brendle, personal communication}, which
also may be a way to solve this problem.

\begin{question}
   Are our main results (specifically, Theorems~\ref{thm:step3},~\ref{altmainstepI},~\ref{mainstep2},~\ref{mainfinal} and~\ref{altmainfinal}) valid for $k_0=1$? I.e., can we force $\mfrak>\aleph_1$?
\end{question}

For $k_0\geq 2$ there was no problem to include, in our iterations, FS products of $k_0$-Knaster posets since they are still $k_0$-Knaster (hence ccc), but we cannot just use FS products of ccc posets because they do not produce ccc posets in general. In particular, we do not know how to modify Theorem~\ref{thm:step3} to force $\mfrak>\aleph_1$.

\begin{question}
   Is it consistent with $\mathrm{ZFC}$ that $\bfrak<\sfrak<\nonM<\covM$?
\end{question}

In this paper $\sfrak\leq\bfrak$; and
forcing $\sfrak>\bfrak$ is much more difficult, since Mathias-Prikry posets may add dominating reals.
Shelah~\cite{SSCR} proved the consistency of $\bfrak=\aleph_1<\sfrak=\cfrak=\aleph_2$ by a countable support iteration of proper posets. Much later, Brendle and Fischer~\cite{VFJB11} constructed an FS iteration via a matrix iteration to force $\aleph_1<\bfrak=\kappa<\sfrak=\cfrak=\lambda$ for arbitrarily chosen regular $\kappa<\lambda$. However, in this latter model, $\nonM=\covM=\cfrak$. It is not clear how to adapt Brendle's and Fischer's methods to our methods and produce a poset for the previous question.

\bibliography{morebib}
\bibliographystyle{amsalpha}


\end{document}

%% file: graph4squashed.pdf_tex
\begingroup%
  \makeatletter%
  \providecommand\color[2][]{%
    \errmessage{(Inkscape) Color is used for the text in Inkscape, but the package 'color.sty' is not loaded}%
    \renewcommand\color[2][]{}%
  }%
  \providecommand\transparent[1]{%
    \errmessage{(Inkscape) Transparency is used (non-zero) for the text in Inkscape, but the package 'transparent.sty' is not loaded}%
    \renewcommand\transparent[1]{}%
  }%
  \providecommand\rotatebox[2]{#2}%
  \newcommand*\fsize{\dimexpr\f@size pt\relax}%
  \newcommand*\lineheight[1]{\fontsize{\fsize}{#1\fsize}\selectfont}%
  \ifx\svgwidth\undefined%
    \setlength{\unitlength}{139.03945442bp}%
    \ifx\svgscale\undefined%
      \relax%
    \else%
      \setlength{\unitlength}{\unitlength * \real{\svgscale}}%
    \fi%
  \else%
    \setlength{\unitlength}{\svgwidth}%
  \fi%
  \global\let\svgwidth\undefined%
  \global\let\svgscale\undefined%
  \makeatother%
  \begin{picture}(1,0.46303914)%
    \lineheight{1}%
    \setlength\tabcolsep{0pt}%
    \put(0,0){\includegraphics[width=\unitlength,page=1]{graph4squashed.pdf}}%
    \put(0.61479296,0.05713701){\color[rgb]{0.85882353,0.2627451,0.14509804}\makebox(0,0)[lt]{\lineheight{1.25}\smash{\begin{tabular}[t]{l}0\end{tabular}}}}%
    \put(0.61375627,0.33710249){\color[rgb]{0.85882353,0.2627451,0.14509804}\makebox(0,0)[lt]{\lineheight{1.25}\smash{\begin{tabular}[t]{l}0\end{tabular}}}}%
    \put(0.2947252,0.33965838){\color[rgb]{0.34117647,0.76862745,0.67843137}\makebox(0,0)[lt]{\lineheight{1.25}\smash{\begin{tabular}[t]{l}1\end{tabular}}}}%
    \put(0.92175668,0.18371614){\color[rgb]{0.34117647,0.76862745,0.67843137}\makebox(0,0)[lt]{\lineheight{1.25}\smash{\begin{tabular}[t]{l}1\end{tabular}}}}%
    \put(-0.00948191,0.18051195){\color[rgb]{0.85882353,0.2627451,0.14509804}\makebox(0,0)[lt]{\lineheight{1.25}\smash{\begin{tabular}[t]{l}0\end{tabular}}}}%
    \put(0,0){\includegraphics[width=\unitlength,page=2]{graph4squashed.pdf}}%
    \put(0.29468569,0.05663461){\color[rgb]{0.34117647,0.76862745,0.67843137}\makebox(0,0)[lt]{\lineheight{1.25}\smash{\begin{tabular}[t]{l}1\end{tabular}}}}%
  \end{picture}%
\endgroup%